\documentclass[11pt, letterpaper]{article}
\usepackage{amsmath, amsthm, amssymb, amsfonts} 
\usepackage{mathrsfs} 
\usepackage{bm} 
\usepackage{graphicx} 
\usepackage{enumerate} 
\usepackage{geometry} 
\usepackage{setspace} 
\usepackage{lmodern} 
\usepackage{hyperref} 
\usepackage{color} 
\usepackage{xcolor} 
\usepackage{url} 
\usepackage{mathtools}
\usepackage{enumitem}
\usepackage{changepage}
\usepackage{microtype}
\usepackage{authblk}
\usepackage{amsfonts}
\usepackage{mathrsfs,amscd,amssymb,amsthm,amsmath,bm,graphicx,psfrag,subfigure,url,mathtools}
\usepackage{pict2e}
\usepackage{psfrag,amsmath}
\usepackage{tikz}
\usepackage{indentfirst}
\usepackage{hyperref}
\usepackage{bookmark}
\usepackage{enumerate}
\usepackage{latexsym,euscript,epic,eepic,color}
\usepackage{multirow}
\usepackage{multicol}
\usepackage{longtable}
\usepackage{adjustbox}

\usepackage{setspace}
\usepackage{epstopdf}
\allowdisplaybreaks
\usepackage{authblk}
\usepackage{pifont}
\numberwithin{equation}{section}

\usepackage{enumitem}
\usepackage{microtype}

\theoremstyle{plain}
\newtheorem{theorem}{Theorem}[section]
\newtheorem{lemma}[theorem]{Lemma}
\newtheorem{proposition}[theorem]{Proposition}

\newtheorem{conjecture}{Conjecture}

\theoremstyle{definition}

\theoremstyle{remark}

\newcommand{\tr}{\operatorname{tr}}
\newcommand{\dist}{\operatorname{dist}}
\newcommand{\1}{\mathbf 1}

\newcommand{\cP}{\mathcal P}

\hypersetup{
    colorlinks=true,
    linkcolor=blue,
    citecolor=red,
    urlcolor=magenta,
}

\newcommand{\keywords}[1]{%
  \par\vspace{6pt}\noindent\textbf{Keywords: }#1\par
}

\newcommand{\MSC}[2][2020]{%
  \par\vspace{3pt}\noindent\textbf{MSC(#1): }#2\par
}

\title{\bf The edge spectral extremal problem for odd wheels
in nonzero residue classes}
\vspace{6mm}
\author{Honghao Chen}
\author{Jing Gao}
\author{Shuchao Li\thanks{Corresponding author. \\
\hspace*{2em}E-mail address: key@mails.ccnu.edu.cn (H. Chen), gjing1270@163.com (J. Gao), lscmath@ccnu.edu.cn (S. Li)} } 

\affil{School of Mathematics and Statistics, and Hubei Key Lab--Math. Sci.,\linebreak Central China Normal University, Wuhan 430079, China}

\date{\today}

\allowdisplaybreaks
\begin{document}
\baselineskip=0.23in

\maketitle

\begin{abstract}
For a fixed integer $k\ge 2$, let $W_{2k+1}=K_1\vee C_{2k}$ be an odd wheel graph. The fixed-size spectral extremal problem aims to determine
\[
  \operatorname{spex}(m,W_{2k+1}):=\max\{\rho(G): e(G)=m,\ G \text{ is } W_{2k+1}\text{-free}\},
\]
where $\rho(G)$ denotes the adjacency spectral radius. Based on this problem, Yu, Li, and Peng~\cite{YLP} proposed the following conjecture: For large $m$, every $W_{2k+1}$-free graph of size $m$ satisfies
\(
 \rho(G)^2-(k-1)\rho(G)\le m-\binom{k}{2}
\)
with equality precisely for $K_k\vee qK_1$ and $m-\binom{k}{2}=kq$.
Very recently, Fang, Zhai and Zhang~\cite{FangZhaiZhang} confirmed the Yu--Li--Peng conjecture. When $m$ is large, $k\ge 2$ and $m-\binom{k}{2}$ is not divisible by $k$, the exact solution for the above problem is still open. Regarding this problem, Yu, Zhang, and Zhang~\cite{YZZ} proposed the following conjecture: Let $r$ be a nonzero remainder when $m-\binom{k}{2}$ is divided by $k$. Then $S_{k,m}$ is the unique graph among $W_{2k+1}$-free graphs of size $m$ having maximum spectral radius, where $S_{k,m}$ is obtained from $K_k\vee qK_1$ by adding a vertex $z$ and joining it to exactly $r$ vertices of the $K_k$. In this paper we address this problem in each nonzero residue class. Our result completely settles the Yu-Zhang-Zhang conjecture for $k\ge 3$. 
\end{abstract}


\MSC{05C50, 05C35}

\section{Introduction}

All graphs in this paper are finite, simple and undirected. We adopt the notation and terminology of graph theory as in \cite{BM2008, Godsil1} unless otherwise stated. For a graph $G$, let $e(G)$ and $\rho(G)$ denote its number of edges and adjacency
spectral radius, respectively.  Since isolated vertices affect neither $e(G)$ nor $\rho(G)$, graphs are understood to be connected. 

The fixed-size spectral extremal problem asks for
\[
 \max\{\rho(G):e(G)=m,\ G\text{ is an } F \text{-free connected graph}\}.
\]
It is the edge-count analogue of the usual fixed-order spectral Tur\'an type
problem and goes back to Brualdi and Hoffman~\cite{BrualdiHoffman}. Research in this direction dates back to 1970, when Nosal~\cite{N1} proved $\rho(G)\le \sqrt{e(G)}$ for every triangle-free graph $G.$ It is generalized by Nikiforov~\cite{E1} who showed that $\rho(G)\le \sqrt{(1-1/r)2e(G)}$ for every $K_{r+1}$-free graph $G$. In recent years, the above fixed-size spectral extremal problem attracted more and more researchers' attention. Exact fixed-size results are known for several forbidden cycles, complete bipartite graphs, fans, friendship graphs, theta graphs, star forest, and related color-critical graphs; see, for example, \cite{Liu-Li-Li-Yu,JoyentanujYamini,LiZhaiShu,LiZhaoZou,ZhaiLinShu,ZhangWang}. For a systematic account of the research progress in this area, we refer the reader to a recent nice survey~\cite{YZZ}.

A $k$-wheel is the graph obtained from joining every vertex of a
$(k-1)$-cycle with an additional central vertex. Since
$W_{2k+1}=K_1\vee C_{2k}$, a graph $G$ is $W_{2k+1}$-free if and only if
$G[N_G(v)]$ is $C_{2k}$-free for every $v\in V(G)$. A $k$-fan is defined as
$F_k=K_1\vee P_{k-1}$.

Yu, Li, and Peng~\cite{YLP} proposed three conjectures concerning the
fixed-size spectral extremal problem for fan graphs and wheel graphs separately.
\begin{conjecture}[Yu-Li-Peng~\cite{YLP}]
\label{conj:1.1}
For fixed $k\ge 2$ and sufficiently large $m$, if $G$ is an $F_{2k+1}$-free
or $F_{2k+2}$-free graph with $m$ edges, then
$\rho(G)\le \frac{k-1+\sqrt{4m-k^{2}+1}}{2}$. Equality holds if and only if
\(
 G\cong K_k\vee qK_1,$ and $m=\binom{k}{2}+kq,
\)
up to isolated vertices.
\end{conjecture}

\begin{conjecture}[Yu-Li-Peng~\cite{YLP}]
\label{conj:1.2}
For fixed $k\ge 2$ and sufficiently large $m$, if $G$ is a $W_{2k+1}$-free
graph with $m$ edges, then $\rho(G)\le \frac{k-1+\sqrt{4m-k^{2}+1}}{2}$. Equality holds if and only if
$G\cong K_k\vee qK_1,$   and  $m=\binom{k}{2}+kq,$
up to isolated vertices.
\end{conjecture}

\begin{conjecture}[Yu-Li-Peng~\cite{YLP}]
\label{conj:1.3}
For fixed $k\ge 2$ and sufficiently large $m$, if $G$ is a $W_{2k+2}$-free
graph with $m$ edges, then $\rho(G)\le \sqrt{4m/3}$. Equality holds if and
only if $G$ is a regular complete $3$-partite graph.
\end{conjecture}

Conjecture~\ref{conj:1.1} was confirmed by Li, Zhao and Zou~\cite{LiZhaoZou} for $k\ge 3$, Gao and Li~\cite{GaoLi} confirmed the case $k=2$. Consequently, they also resolved the corresponding problem for friendship
graphs. Conjecture~\ref{conj:1.3} was settled affirmatively by Li, Liu and
Zhang~\cite{LiLiuZhangStability}. Very recently, 
Fang, Zhai and Zhang~\cite{FangZhaiZhang} established the following sharp universal bound.
\begin{theorem}[Fang--Zhai--Zhang \cite{FangZhaiZhang}]
\label{thm:FZZ}
Fix $k\ge3$.  Every sufficiently large $m$-edge
$W_{2k+1}$-free graph $G$ satisfies
\(
 \rho(G)^2-(k-1)\rho(G)\le m-\binom{k}{2}.
\)
Equality holds if and only if
\(
 G\cong K_k\vee qK_1\), $and$ \(m=\binom{k}{2}+kq,
\)
up to isolated vertices.
\end{theorem}

Note that when $m-\binom{k}{2}$ is not divisible by $k$, equality in
Theorem~\ref{thm:FZZ} is impossible.  Fang, Zhai and Zhang~\cite{FangZhaiZhang} explicitly
asked for the exact maximum in these nonzero residue classes. Interestingly, it so happens that Yu, Zhang and Zhang~\cite{YZZ} proposed a conjecture for this open problem as follows.
\begin{conjecture}[Yu-Zhang-Zhang~\cite{YZZ}]\label{conj:1.4}
Let $k\ge 2$ be fixed and $m$ be sufficiently large, and let $G$ be a
$W_{2k+1}$-free graph with $m$ edges.
If $m=\binom{k}{2}+kq+r$ with $1\le r\le k-1$, then
\(
\rho(G)\le \rho(S_{k,m}),
\)
equality holds if and only if
$G\cong S_{k,m}$, where $S_{k,m}$ is obtained from $K_k\vee qK_1$ by adding a vertex $z$ and joining it to exactly $r$ vertices of the
$K_k$.
\end{conjecture}

In this paper, we resolve Conjecture~\ref{conj:1.4} for $k\ge 3.$ The $k=2$ case is by far the more involved one: its extremal graphs depend
genuinely on the arithmetic of $m$, and the second-order analysis requires quantizing two matching deficiencies. To this end, we will specifically treat the case $k=2$ in a standalone article~\cite{LGY}. Our main contribution of this paper can be formulated as follows:
\begin{theorem}
\label{thm:main}
For every fixed integer $k\ge3$, there exists $m_0=m_0(k)$ such that
the following holds.  If $m\ge m_0$ and $G$ is an $m$-edge
$W_{2k+1}$-free connected graph, then
\(
 \rho(G)\le \rho(S_{k,m}).
\)
Equality holds if and only if $G\cong S_{k,m}$.
\end{theorem}

\noindent{\bf Our approach.}\ \ Our proof strategy is as follows. To prove Theorem~\ref{thm:main}, we proceed in three steps.

First, the ordinary dense-core iteration cannot be applied directly to
$\rho(S_{k,m})$: after dividing by $\sqrt m$, the coefficient of $m^{-1}$
depends on $m-\binom{k}{2}\pmod k$, so one-edge deletions create oscillations of
the same order as the desired gain.  We introduce a
\emph{congruence-preserving core}, in which every allowed deletion has
size divisible by $k$.

Second, we combine the resulting Perron estimates with the stability and
residual framework of~\cite{FangZhaiZhang}.  A bounded number of light
edges may survive; this is essential because the partially adjacent
vertex of $S_{k,m}$ itself may be incident with light edges.  We show
that the graph nevertheless consists of a bounded set $K$ and an
independent set $K'$ satisfying $N(K')\subseteq K$.

Third, we partition $K'$ into independent twin classes.  An exact matrix
identity forces all but $O_k(1)$ vertices of $K'$ into one neighborhood
type.  A resolvent expansion then converts the problem to a finite
optimization.  The final residue comparison follows from the elementary
modular inequality; see Lemma~\ref{lem:modular} below.\vspace{2mm}

\noindent{\bf Organization.}\ \ In Section \ref{s2}, we give some preliminaries. In Section \ref{s3}, we give the proof of Theorem~\ref{thm:main}, which consists of the candidate and its spectral threshold, congruence-preserving dense cores, stability partition and Perron localization, the bounded-core reduction, and the finite-dimensional closure. Some concluding remarks are given in the last section.

\section{Preliminaries}\label{s2}

For $S\subseteq V(G)$, write $e(S)=e(G[S])$.  For two disjoint vertex subsets $S,T\subseteq V(G)$, let
\(
E(S,T)=\{uv\in E(G)\mid u\in S,\,v\in T\},
\)
and $e(S,T)=|E(S,T)|$ denote the number of edges between $S$ and $T$. We use
$d_S(v)=|N(v)\cap S|$, and  use $\Delta(G)$ to denote maximum degree of $G.$
Let \(\boldsymbol{x}\) be the chosen unite perron vector and \(x_{u^*}=\max_{v\in V(H)}x_v\).
We have the following Perron mass stability result.

As usual,
we use $K_n$ and $C_n$ to denote the complete graph and cycle on $n$ vertices, respectively.

Let $G$ be a connected simple graph of order $n$, and let $A$ be its adjacency matrix.
Let $\rho=\rho(A)$ be the spectral radius, i.e., the largest real eigenvalue of $A$.
From the Perron--Frobenius theorem for irreducible non-negative matrices,
there exists a real vector $\boldsymbol{x}\in\mathbb{R}^n$ with all entries positive satisfying
\(
A\boldsymbol{x} = \rho\,\boldsymbol{x}.
\)
This positive eigenvector $\boldsymbol{x}$ associated with $\rho$ is called the \emph{Perron vector} of graph $G$.
As usual, a nonnegative unit eigenvector belonging to
$\rho(G)$ is denoted by $\boldsymbol{x}=(x_v)_{v\in V(G)}$.  If $G$ is disconnected,
we choose $\boldsymbol{x}$ positive on a component attaining $\rho(G)$ and zero on
all other components.

For two disjoint vertex sets $A, B,$ we write $K_{A,B}$ for the complete bipartite graph on the parts $A$
and $B$.
We use the following consequence of edge-spectral stability result
\cite{LiLiuZhangStability}.
\begin{lemma}[\cite{LiLiuZhangStability}]
\label{lem:stability}
Fix a graph $F$ with $\chi(F)=3$.  If $G$ is an $F$-free graph with
$m\to\infty$ edges and there exists a constant \(\delta>0\) such that
\(
 \rho(G)\geq\sqrt{ (1-\delta)m},
\)
then there are disjoint sets $A,B\subseteq V(G)$ such that
\(
 \dist(G,K_{A,B})=o(m).
\)
Moreover, \(\rho(G)\leq \sqrt{(1+o(1))m}\). Here the distance is the number of edge additions and deletions required
to transform one graph into the other, with all remaining vertices
isolated.
\end{lemma}

The next exact identity is the residual identity used in
\cite{FangZhaiZhang,ZhaiLiLou}.  We include its short derivation.

\begin{lemma}
\label{lem:residual}
Let $G$ have $m$ edges and spectral radius $\rho$, and choose
$u^*$ with $x_{u^*}=\max_{v\in V(G)}x_v$.  Put
\(
 U=N(u^*), W=V(G)\setminus N[u^*]
\)
and
\(
 f(w)=d_U(w)(x_{u^*}-x_w)+\frac12d_W(w)x_{u^*}\) for \(w\in W\).
Then, for arbitrary real constants $c,d$, one has
\begin{align}
 (\rho^2-d\rho+c-e(G))x_{u^*}
 &=\sum_{uv\in E(U)}(x_u+x_v-x_{u^*})
 -(d\rho-c)x_{u^*} -\sum_{w\in W}f(w).
\label{eq:residual-general}
\end{align}
\end{lemma}

\begin{proof}
The two-step eigen-equation at $u^*$ gives
\begin{equation}\label{e:002}
 \rho^2x_{u^*}=|U|x_{u^*}+
 \sum_{uv\in E(U)}(x_u+x_v)+
 \sum_{w\in W}d_U(w)x_w.
\end{equation}
Since
\(
 |U|+e(U,W)+e(W)=e(G)-e(U),
\) \(e(U,W)=\sum_{w\in W}{d_U(w)}\) and \(e(W)=\frac{1}{2}\sum_{w\in W}{d_W(w)}\),
\eqref{e:002} is equivalent to
\[
 \rho^2x_{u^*}=e(G)x_{u^*}+
 \sum_{uv\in E(U)}(x_u+x_v-x_{u^*})-
 \sum_{w\in W}f(w).
\]
Subtracting $(d\rho-c+h)x_{u^*}$ proves
\eqref{eq:residual-general}.
\end{proof}

For any real symmetric matrix $M$, its \emph{spectral norm} is
\(
\|M\|
=\max_{\|\boldsymbol{x}\|=1}\|M\boldsymbol{x}\|
=\max\bigl\{|\lambda|\,\big|\) $\lambda$ is an eigenvalue of $M\bigr\}$.
The \emph{Frobenius norm} is defined by
\(
\|M\|_F
=\sqrt{\sum_{i,j} |M_{ij}|^2}.
\)
The following spectral norm inequality for matrices will be repeatedly used in the subsequent sections.
\begin{lemma}
\label{lem:frobenius}
Let $M$ be a real symmetric matrix with exactly $2D$ nonzero entries,
each of absolute value one.  Then
\(
 \|M\|\le\|M\|_{ F}=\sqrt{2D},
\)
where $\|\cdot\|$ denotes the spectral norm (operator $2$-norm), and $\|\cdot\|_F$ denotes the Frobenius norm.
\end{lemma}
\begin{proof}
The inequality \(\|M\| \le \|M\|_F\) follows by \cite[Theorem 5.6.26(b)]{H-J-2012}. 


By assumption, $M$ has at most $2D$ nonzero entries, and every nonzero entry satisfies $|M_{ij}|=1$.
Each nonzero entry contributes $|M_{ij}|^2 = 1$ to the Frobenius-norm squared; zero entries contribute $0$.
Therefore
\[
\|M\|_F^2
=\sum_{i,j}|M_{ij}|^2
\le \sum_{\substack{(i,j)\\M_{ij}\neq 0}} 1
\le 2D.
\]
Taking square roots yields
\(
\|M\|_F \le \sqrt{2D}.
\)
Combining with $\|M\|\le\|M\|_F$ gives us
\(
\|M\| \le \|M\|_F \le \sqrt{2D},
\)
as desired.
\end{proof}
\section{Proof of Theorem~\ref{thm:main}}\label{s3}
\subsection{The candidate and its spectral threshold}

Throughout this section, fix
\[
 m=\binom k2+kq+r,\qquad 1\le r\le k-1,
\]
and write
\[
 s_k(m):=\rho(S_{k,m}).
\]

\begin{proposition}
\label{prop:candidate}
The graph $S_{k,m}$ is $W_{2k+1}$-free.  Moreover, $s_k(m)$ is the
largest root of
\begin{align}
 F_{k,q,r}(x)
 :={}&x(x+1)\bigl(x^2-(k-1)x-kq-r\bigr)+r(k-r)(x+q).
\label{eq:quartic}
\end{align}
\end{proposition}

\begin{proof}
Let $C$ be the $k$-clique, $Q$ the independent set of size $q$, and
$z$ the partially adjacent vertex.  If a proposed wheel center lies in
$Q$, its neighborhood is $C$ and has fewer than $2k$ vertices.  The
neighborhood of $z$ is a clique of order $r<k$.  If the center lies in
$C$, the nonclique vertices of its neighborhood form an independent
set, while the clique part has only $k-1$ vertices.  Every cycle in such
a graph contains at most as many independent vertices as clique
vertices, so it has length at most $2(k-1)$.  Thus $S_{k,m}$ is
$W_{2k+1}$-free.

Partition $C$ into the $r$ neighbors and the $k-r$ nonneighbors of $z$.
Together with $Q$ and $\{z\}$ this is an equitable partition with
quotient matrix
\[
 \begin{pmatrix}
 r-1&k-r&q&1\\
 r&k-r-1&q&0\\
 r&k-r&0&0\\
 r&0&0&0
 \end{pmatrix}.
\]
Its characteristic polynomial is~\eqref{eq:quartic}. The Perron-Frobenius Theorem shows that the largest root is $s_k(m)$.
\end{proof}

In our whole context, let $c_k=\binom k2$. Define the residual
\[
 \cP_h(x):=x^2-(k-1)x-h+c_k.
\]
\begin{lemma}[\cite{Liu-Li-Li-Yu}]\label{lem2.4}
Let $m, k, a, b$ be positive integers with $m\geq 144k^4$ and $k\geq 3$. If $m + \frac{k(k+1)}{2} \equiv r \pmod{k}$ with $1\le r\le k-1$, then $\rho (S_{k,m})> \frac{k-1 + \sqrt{4m -k^2 -k +1}}{2} > \sqrt{m}$ .
\end{lemma}

\begin{lemma}
\label{lem:candidate-expansion}
For $s=s_k(m)$,
\begin{equation}
s^2-(k-1)s-h+c_k =\cP_m(s)
 =-\frac{r(k-r)(s+q)}{s(s+1)}.
\label{eq:candidate-residual}
\end{equation}
Consequently, $-c_k\le\cP_m(s)<0$.  Moreover,
\begin{equation}
 s_k(m)=\sqrt m+\frac{k-1}{2}
 -\left(\frac{k^2-1}{8}+\frac{r(k-r)}{2k}\right)m^{-1/2}
 +O_k(m^{-1}).
\label{eq:candidate-expansion}
\end{equation}
\end{lemma}
\begin{proof}
Since $m-c_k=kq+r$, equation~\eqref{eq:quartic} at $x=s$ reads
\[
 s(s+1)\cP_m(s)+r(k-r)(s+q)=0,
\]
which proves~\eqref{eq:candidate-residual}.  In particular the residual
is negative and bounded in absolute value by a constant depending only
on $k$.

On the other hand, by  Theorem \ref{thm:FZZ} and Lemma \ref{lem2.4}. We obtain

\[
  \frac{k-1 + \sqrt{4m -k^2  +1}}{2}> s > \frac{k-1 + \sqrt{4m -k^2 -k +1}}{2}.
\]
So, \(\cP_m(s)=s^2-(k-1)s-m+c_k>-\frac{k}{4}\geq -c_k\) by \(s>\frac{k-1 + \sqrt{4m -k^2 -k +1}}{2}\).
Furthermore, we have
\[
 s-\sqrt m=a+b m^{-1/2}+o(m^{-\frac{1}{2}}), \text{  i.e.,  }  s=\sqrt m+a+b m^{-1/2}+o(m^{-\frac{1}{2}}).
\]
Substituting \(s\) into 
\eqref{eq:candidate-residual}, one sees that the constant term gives
$2a=k-1$, whereas the coefficient of $m^{-1/2}$ gives
\[
 a^2+2b-(k-1)a=-\frac{r(k-r)}{k}.
\]
Using $a=(k-1)/2$ yields
\[
 b=-\frac{k^2-1}{8}-\frac{r(k-r)}{2k}.
\]

Next, we derive finer estimates for the residual quantity of $s$. Define the small parameter
\(
y=m^{-1/2}\), i.e., \(m=\frac{1}{y^2}.
\)
Then
\[
q=\frac{m-\binom{k}{2}-r}{k}=\frac{1}{k y^2}-\frac{1}{k}\left(\binom{k}{2}+r\right).
\]
Substitute $m,q$ in terms of $y$ into the polynomial equation $F_{k,q,r}(s)=0$.
Rewrite the equation as an analytic implicit equation
\(
G(y,s)=0,
\)
where $G(y,s)$ is a real-analytic function in two variables near $(y,s)=(0,+\infty)$.

For $y\to 0$ (i.e.\ $m\to\infty$), the leading-order solution satisfies $s\sim 1/y=\sqrt{m}$.
Evaluate the partial derivative with respect to $s$ at the leading-order root:
\[
\frac{\partial G}{\partial s}\bigg|_{(y=0,s=1/y)}\neq 0.
\]
By the analytic implicit-function theorem, there exists a unique real-analytic solution $s=s(y)$ defined for sufficiently small $y>0$, which admits a convergent power-series expansion
\[
s(y)=\frac1y + b\, y + a  + O(y^2).
\]
Recall $y=m^{-1/2}$, $ a=\frac{k-1}{2}$ and $b = -\frac{k^2-1}{8}-\frac{r(k-r)}{2k}$. Substitute back:
\[
s=s(y)=\sqrt{m}+\frac{k-1}{2}-\left(\frac{k^2-1}{8}+\frac{r(k-r)}{2k}\right)m^{-1/2}+O_k(m^{-1}).
\]
The bound $O_k(\cdot)$ means the remainder depends only on fixed $k,r$ and holds uniformly for all sufficiently large $m$.
This completes the proof.
\end{proof}

For real $t$ with $t$ sufficiently large, keep the residue $r$ fixed,
put $q(t)=(t-c_k-r)/k$, and $s_{k,r}(t)$ denotes the largest real root of the polynomial $F_{k,q(t),r}(s)=0$, where $q$ is replaced by the function $q(t)$.

\begin{lemma}
\label{lem:normalized-derivative}
For fixed $k$ and $r$,
\(
 \left|\frac{d}{dt}\frac{s_{k,r}(t)}{\sqrt t}\right|
 \le c_k t^{-3/2}
\)
for all sufficiently large $t$.
\end{lemma}
\begin{proof}
By Lemma \ref{lem:candidate-expansion}, we have
\[ s_k(m)=\sqrt m+\frac{k-1}{2}
 -\left(\frac{k^2-1}{8}+\frac{r(k-r)}{2k}\right)m^{-1/2}
 +O_k(m^{-1}).\]
 Recall \(m= c_k+kq+r\) and \(t= c_k+kq(t)+r\), so that  \(s_{k,r}(t)= s_k(t)\). Combining this with large \(t\) yields
\[
 \frac{d}{dt}\frac{s_{k,r}(t)}{\sqrt t}
 =-\frac{k-1}{4}t^{-3/2}+O_k(t^{-2}),
\]
which proves the lemma.
\end{proof}

\subsection{Congruence-preserving dense cores}

For a graph $G$ with at least one edge, set
\[
 \Phi(G):=\frac{\rho(G)}{\sqrt{e(G)}}.
\]
Fix $0<\varepsilon<10^{-3}$.  A subgraph is always allowed to discard
isolated vertices after edges have been deleted.

Let $G$ be a graph. Based on $\varepsilon$-\textit{dense} proposed by Fang, Lin and Zhai~\cite{FangLinZhai}, a proper edge-induced subgraph $G'\subsetneq G$
is $(\varepsilon,k)$-\textit{dense} in $G$ if 
\(0<e(G)-e(G')<\varepsilon e(G) 
 ,\, k\mid (e(G)-e(G'))
\)
and
\(
 \Phi(G')-\Phi(G)\ge\frac{\varepsilon(e(G)-e(G'))}{2e(G)}.
\)
A graph with no such proper subgraph is called an
$(\varepsilon,k)$-\textit{core}.
The next lemma is a core extraction.
\begin{lemma}
\label{lem:core-extraction}
Let $G$ be a sufficiently large $m$-edge $W_{2k+1}$-free graph with
\(
 \rho(G)\ge s_k(m),\ m-c_k\equiv r\pmod k.
\)
Then $G$ contains an $(\varepsilon,k)$-core $H$ such that, with
$h=e(H)$,
\[
 h=(1-o(1))m,\qquad h\equiv m\pmod k,
 \qquad \rho(H)\ge s_k(h).
\]
If $H\ne G$, then $\rho(H)>s_k(h)$.
\end{lemma}

\begin{proof}
Starting with $G_0=G$, whenever $G_i$ is not a core choose an
$(\varepsilon,k)$-dense proper subgraph $G_{i+1}$.  Put
\(
 m_i=e(G_i),\ \zeta_i=m_i-m_{i+1}.
\)
The process terminates at a core $H=G_t$, and every $\zeta_i$ is
divisible by $k$.  Hence $h\equiv m\pmod k$.

Summing the defining inequalities gives
\begin{equation}
 \Phi(H)-\Phi(G)
 \ge\frac{\varepsilon}{2}
 \sum_{i=0}^{t-1}\frac{\zeta_i}{m_i}.
\label{eq:core-sum}
\end{equation}
Set \(x_i = \zeta_i/m_i\). Then we have \(0<x_i<\varepsilon\) and \(m_{i+1} = m_i(1 - x_i)\). Thus,
\[
\log \frac{m_0}{m_t} = \sum_{i=0}^{t-1} \log \frac{m_i}{m_{i+1}} = \sum_{i=0}^{t-1} -\log(1 - x_i).
\]
For \(x \in [0, \varepsilon]\), we have \(-\log(1 - x) = \int_0^x \frac{dt}{1-t} \le \frac{x}{1-\varepsilon}\). Since \(x_i \in (0, \varepsilon)\), we obtain
\[
\sum_{i=0}^{t-1} \frac{\zeta_i}{m_i} = \sum_{i=0}^{t-1} x_i \ge (1 - \varepsilon) \sum_{i=0}^{t-1} -\log(1 - x_i) = (1 - \varepsilon) \log \frac{m}{m_t}.
\]
Together with (\ref{eq:core-sum}), one has
\begin{equation}
\Phi(H) - \Phi(G) \ge \frac{\varepsilon(1 - \varepsilon)}{2} \log \frac{m}{m_t}.
\label{eq:4}
\end{equation}
We know that \(\rho^2(H)\leq tr(A(H)^2)=2e(H)\), and therefore $\Phi(H)\le\sqrt2$, while \(\Phi(G)=\frac{s_k(m)}{\sqrt{e(G)}}=1+o(1)\). Thus, (\ref{eq:4}) gives \(\log \frac{m}{m_t}=O_{\varepsilon}(1)\) and so $h=m_t\ge c_{\varepsilon} m$ for some \(c_{\varepsilon}>0\).

If $h\le(1-\alpha)m$ for a fixed $\alpha>0$, then
\eqref{eq:core-sum} yields
\[
 \Phi(H)-\Phi(G)\geq \frac{\varepsilon}{2m}\sum_{i=0}^{t-1}{\zeta_i}\geq\frac{\varepsilon(m-h)}{2m}
 \ge\frac{\varepsilon\alpha}{2}.
\]
Since $H$ is $W_{2k+1}$-free and $h\ge c_\varepsilon m$, Lemma \ref{lem:stability} gives $\Phi(H)\le1+o(1)$, a contradiction.  Therefore
$h=(1-o(1))m$.

We may now assume $h\ge m/2$.  By
Lemma~\ref{lem:normalized-derivative} and Lagrange Mean Value Theorem, we obtain
\[
 \left|
 \frac{s_k(h)}{\sqrt h}-\frac{s_k(m)}{\sqrt m}
 \right|
 \le c_k\frac{m-h}{m^{3/2}},
\]
where the same residue branch is used because $h\equiv m\pmod k$.
On the other hand,~\eqref{eq:core-sum} gives
\[
 \Phi(H)-\Phi(G)\ge\frac{\varepsilon(m-h)}{2m}.
\]
For large $m$, the latter dominates the former whenever $m-h>0$.
Since $\Phi(G)\ge s_k(m)/\sqrt m$, we obtain\[
\Phi(H)-\ \frac{s_k(h)}{\sqrt h}=\bigl(\Phi(G)-\frac{s_k(m)}{\sqrt m}\bigr)+\bigl(\Phi(H)-\Phi(G)\bigr)-\bigl( \frac{s_k(h)}{\sqrt h}-\frac{s_k(m)}{\sqrt m}\bigr)>0.
\]
So that $\rho(H)\ge s_k(h)$, with strict inequality if $H\ne G$.
\end{proof}

The two lemmas below serve primarily to estimate the Perron components. The first asserts that there are at most \(k-1\) edges having a negligible influence on the Rayleigh quotient. The second provides a lower bound on the square of the Perron component of a vertex with degree restriction.
\begin{lemma}
\label{lem:light-edges}
Let $H$ be a sufficiently large $(\varepsilon,k)$-core of size $h$ with $\rho(H)\ge\sqrt h$, and unit Perron vector $\boldsymbol{x}$.  Then fewer than $k$
edges $uv$ satisfy
\begin{equation}
 x_ux_v\le\frac{1-4\varepsilon}{4\sqrt h}.
\label{eq:light}
\end{equation}
\end{lemma}
\begin{proof}
Suppose to the contrary that there exist \(k\) edges satisfying ~\eqref{eq:light}, and let \(E_0\) be the set of these \(k\) edges. let $H'=H-E_0$.  The Rayleigh quotient gives
\[
 \rho(H')\ge\rho(H)-2\sum_{uv\in E_0}x_ux_v
 \ge\rho(H)-\frac{k(1-4\varepsilon)}{2\sqrt h}.
\]
Therefore
\begin{align*}
 \frac{\Phi(H')}{\Phi(H)}
 &\ge
 \left(1-\frac{k(1-4\varepsilon)}{2\rho(H)\sqrt h}\right)
 \left(1-\frac{k}{h}\right)^{-1/2}
 \ge
 \left(1-\frac{k(1-4\varepsilon)}{2h}\right)
 \left(1+\frac{k}{2h}\right)
 >1+\frac{\varepsilon k}{h}
\end{align*}
for sufficiently large $h$.  As $\Phi(H)\ge1$, this implies
\(
 \Phi(H')-\Phi(H)>\frac{\varepsilon k}{2h},
\)
contrary to the $(\varepsilon,k)$-core property.
\end{proof}

\begin{lemma}
\label{lem:degree-coordinate}
Under the assumptions of Lemma~\ref{lem:light-edges}, every vertex $u$
with $d=d_H(u)\le\varepsilon h/8$ satisfies
\begin{equation}
 2h x_u^2\ge(1-20\varepsilon)d-8k.
\label{eq:degree-coordinate}
\end{equation}
\end{lemma}

\begin{proof}
Choose $\ell\in\{0,\ldots,k-1\}$ so that $k\mid (d+\ell)$.
Among the $h-d$ edges not incident with $u$, take $\ell$ edges $e_1,\ldots,e_\ell$ with
minimum total Perron product.  Since
$\sum_{vw\in E(H)}x_vx_w=\rho(H)/2$, one has
$$
  \sum_{vw\in \{e_1,\ldots, e_\ell\}}x_vx_w\le \frac{\ell\rho(H)}{2(h-d)}.
$$

Delete these edges together with all $d$ edges incident with $u$, and
write $H'$ for the resulting graph.  Put $t=x_u^2$ and
$z=(d+\ell)/h$.  Restricting $\boldsymbol{x}$ to $V(H)\setminus\{u\}$ and
renormalizing gives
\begin{equation}
 \frac{\Phi(H')}{\Phi(H)}
 \ge
 \frac{1-2t-\ell/(h-d)}{(1-t)\sqrt{1-z}}.
\label{eq:degree-ratio}
\end{equation}

If the right-hand side of~\eqref{eq:degree-coordinate} is nonpositive,
there is nothing to prove.  Otherwise suppose, for a contradiction, that
\(
 t<\frac{(1-20\varepsilon)d-8k}{2h}.
\)
Then $t\le\varepsilon/16$, $z\le\varepsilon/7$ for large $h$, and
$\ell/(h-d)\le2k/h$.  Using
$(1-z)^{-1/2}\ge1+\frac{1}{2}z$ and
$(1-t)<1$, a direct multiplication in
\eqref{eq:degree-ratio} yields
\begin{equation}
 \frac{\Phi(H')}{\Phi(H)}-1
 \ge
 \frac{d+\ell}{2h}-(1+3\varepsilon)
 \left(t+\frac{2k}{h}\right).
\label{eq:degree-ratio-lower}
\end{equation}
Substituting the assumed upper bound on $t$ into
\eqref{eq:degree-ratio-lower}, and using $\ell\le k-1$, gives
\[
 \frac{\Phi(H')}{\Phi(H)}-1
 \ge\frac{\varepsilon(d+\ell)}{2h}
\]
for $\varepsilon<10^{-3}$ and all sufficiently large $h$.
Since $d+\ell$ is positive, divisible by $k$, and less than
$\varepsilon h$, this contradicts the definition of an
$(\varepsilon,k)$-core.
\end{proof}

\subsection{Stability partition and Perron localization}
By Lemma \ref{lem:core-extraction}, suppose that $H$ is an $(\varepsilon,k)$-core of $G$. Then $H$ is a $W_{2k+1}$-free $(\varepsilon,k)$-core with
$h\to\infty$ edges and
\(
 \rho:=\rho(H)\ge s_k(h).
\)
Choose constants in the hierarchy
\(
 0<\varepsilon\ll\eta^4\ll k^{-1}.
\)
By Theorem~\ref{thm:FZZ} and
Lemma~\ref{lem:candidate-expansion},
\(
 \rho=(1+o(1))\sqrt h.
\)

Based on Lemma~\ref{lem:stability}, among all pairs of disjoint sets
$A,B$, choose one minimizing
\(
 D:=\dist(H,K_{A,B}),
\)
and, subject to this, minimizing
$R:=V(H)\setminus(A\cup B)$.  Write
\(
 a=|A|\le b=|B|.
\)
Then
\begin{equation}
 D=o(h),\qquad ab=(1+o(1))h.
\label{eq:ab}
\end{equation}
An edge between $A$ and $B$ is called a \emph{missing cross edge} if it lies in the complete bipartite graph $K_{A,B}$ but is missing from $H$.

The minimality of \(D\) and \(|R|\) implies that
\begin{equation}
 \begin{aligned}
 d_B(u)&\ge b/2 && \text{for } u\in A,\\
 d_A(v)&\ge a/2 &&\text{for } v\in B ,\\
 d_A(w)&<a/2,\quad d_B(w)<b/2 && \text{for } w\in R.
 \end{aligned}
\label{eq:minimality}
\end{equation}

Indeed, if $u\in A$ satisfies $d_B(u)<b/2$, then removing $u$ from $A$ replaces $b-d_B(u)$ missing cross edges with only $d_B(u)$ extra edges, which strictly decreases $D$. The other assertions follow in the same way by moving a vertex into or out of a part.

For fixed $\gamma>0$, put
\[
 A_\gamma=\{u\in A:d_B(u)\ge(1-\gamma)b\},\qquad
 B_\gamma=\{v\in B:d_A(v)\ge(1-\gamma)a\}.
\]
The number of missing cross edges is $D=o(h)=o(ab)$, so
\begin{equation}
 |A\setminus A_\gamma|=o(a),\qquad
 |B\setminus B_\gamma|=o(b).
\label{eq:good-sets}
\end{equation}

Write $A^*=A_{\eta^2}$ and $B^*=B_{\eta^2}$. For the above graphs $H$ and $K_{A,B}$, we call an edge $uv$ to be \textit{edit-error} if the edge $uv\in  (E(H)\setminus E(K_{A,B}))\cup(E(K_{A,B})\setminus E(H))$.

Let \(\boldsymbol{x}\) be the chosen unite perron vector and \(x_{u^*}=\max_{v\in V(H)}x_v\). We have the following Perron mass stability result.
\begin{lemma}
\label{lem:4.2}
There is a constant $C>0$ such that, for all sufficiently large $h$,
\[
x_u^2 \ge (1-C\eta^2)\frac{b}{2h}\quad (u\in A^*) \quad\text{and}\quad
x_v^2 \ge (1-C\eta^2)\frac{a}{2h}\quad (v\in B^*).
\]
Consequently, both $\sum_{u\in A^*}x_u^2$ and $\sum_{v\in B^*}x_v^2$ are bounded below by $\frac12\bigl(1-C\eta^2-o(1)\bigr)$.
\end{lemma}
\begin{proof}
Take $v\in B^*$. Then $d_A(v)\ge (1-\eta^2)a$.
Since $a^2\le ab=(1+o(1))h$, it follows that $a=O(\sqrt{h})=o(h)$.
Recall $\varepsilon\ll \eta^4\ll 1$. Every neighbor of $v$ outside $A$ is incident to an edit-error edge, so $d_H(v)\le a+D<\frac{\varepsilon}{8} h$ for large $h$. Lemma \ref{lem:degree-coordinate} gives
\begin{equation}
x_v^2 \ge (1-20\varepsilon)\frac{d_H(v)-8k}{2h}
\ge (1-2\varepsilon)(1-\eta^2)\frac{a}{2h}
\ge (1-3\eta^2)\frac{a}{2h}.
\label{eq:22}
\end{equation}

Now take $u\in A^*$. Then $d_B(u)\ge (1-\eta^2)b$.
Since $|B\setminus B^*|=o(b)$ by (\ref{eq:good-sets}), we have $d_{B^*}(u)\ge (1-2\eta^2)b$.
Applying the lower bound from \eqref{eq:22} to each $v\in N_{B^*}(u)$ and invoking the eigen-equation at $u$, we arrive at
\[
\rho x_u \ge \sum_{v\in N_{B^*}(u)} x_v
\ge (1-2\eta^2)b\sqrt{(1-3\eta^2)\frac{a}{2h}}.
\]
Squaring and using \(ab=(1+o(1))h\) together with $\rho^2=(1+o(1))h$, we further obtain
\begin{equation}
x_u^2 \ge (1-C\eta^2)\frac{b}{2h}
\label{eq:23}
\end{equation}
for some constant $C\ge 3$. Combining \eqref{eq:22} and \eqref{eq:23} establishes the first assertion.

Summing \eqref{eq:22} over $B^*$ and \eqref{eq:23} over $A^*$, and using the estimates
$|A^*|=(1-o(1))a$, $|B^*|=(1-o(1))b$ together with $ab=(1+o(1))h$, we obtain
\begin{equation}
\min\left\{\sum_{u\in A^*}x_u^2,\;\sum_{v\in B^*}x_v^2\right\}
\ge (1-C\eta^2)(1-o(1))\frac{ab}{2h}
\ge \frac12\bigl(1-C\eta^2-o(1)\bigr),
\label{eq:24}
\end{equation}
as desired.
\end{proof}
\begin{lemma}
\label{lem:mass-stability}
 Uniformly for all
$S\subseteq A$ and $T\subseteq B$,
\begin{equation}
 \sum_{u\in S}x_u^2\le\frac{|S|}{2a}+o(1),\qquad
 \sum_{v\in T}x_v^2\le\frac{|T|}{2b}+o(1),
\label{eq:subset-mass}
\end{equation}
and
\begin{equation}
 \sum_{w\in R}x_w^2=o(1).
\label{eq:R-mass}
\end{equation}
\end{lemma}

\begin{proof}
Let $J$ be $K_{A,B}$ together with $|R|$ isolated vertices.  By
Lemma~\ref{lem:frobenius},
\[
 \|A(H)-A(J)\|\le\sqrt{2D}=o(\sqrt h)=o(\sqrt{ab}).
\]
Let
\(
 \boldsymbol{y}=\frac{1}{\sqrt{2a}}\1_A+
   \frac{1}{\sqrt{2b}}\1_B,
\)
where \[
    (\boldsymbol{1}_A)_v=
    \begin{cases}
        1, & v\in A,\\
        0, & v\notin A,
    \end{cases}
    \ \ \ \
    (\boldsymbol{1}_B)_v=
    \begin{cases}
        1, & v\in B,\\
        0, & v\notin B.
    \end{cases}
    \]
The nonzero eigenvalues of $A(J)$ are $\sqrt{ab}$ and $-\sqrt{ab}$.
Since
\[
 \boldsymbol{x}^{\mathsf T}A(J)\boldsymbol{x}
 = \boldsymbol{x}^{\mathsf T}A(H)\boldsymbol{x}- \boldsymbol{x}^{\mathsf T}(A(H)-A(J))\boldsymbol{x}\ge\rho(H)-\|A(H)-A(J)\|
 \geq(1-o(1))\sqrt{ab}.
\]
Let $\boldsymbol{y'}=\frac{1}{\sqrt{2a}}\mathbf{1}_A-\frac{1}{\sqrt{2b}}\mathbf{1}_B$. Then $\boldsymbol{y}$ is the unit eigenvector for $\lambda_1=\sqrt{ab}$, and $\boldsymbol{y'}$ is the unit eigenvector for $\lambda_2=-\sqrt{ab}$. The eigenspace of the zero eigenvalue is orthogonal to $\boldsymbol{y}$ and $\boldsymbol{y'}$. Therefore, $A(J)$ admits the spectral decomposition
\[
A(J) = \sqrt{ab}\, \boldsymbol{y}\boldsymbol{y}^{\mathsf T} - \sqrt{ab}\, \boldsymbol{y'}\boldsymbol{y'}^{\mathsf T} + 0\cdot \bigl(\text{orthogonal projection onto } \ker A(J)\bigr).
\]
Here, \(\ker A(J) = \{\boldsymbol{v} : A(J)\boldsymbol{v} = \boldsymbol{0}\}.\)

Decompose \(\boldsymbol x\) with respect to \( \boldsymbol y\), \( \boldsymbol{y'}\), and $\ker A(J)$:
\[
 \boldsymbol x = ( \boldsymbol y^{\mathsf T}  \boldsymbol x) \boldsymbol y + ( \boldsymbol{y'}^{\mathsf T}  \boldsymbol x) \boldsymbol{y'} +  \boldsymbol {x_0},
\]
where \( \boldsymbol {x_0} \in \ker A(J)\), and
\(
\| \boldsymbol x\|^2 = ( \boldsymbol y^{\mathsf T}  \boldsymbol x)^2 + ( \boldsymbol y'^{\mathsf T}  \boldsymbol x)^2 + \| \boldsymbol{x_0}\|^2 = 1.
\)
Then
\(
 \boldsymbol x^{\mathsf T} A(J) \boldsymbol x = \sqrt{ab}\,( \boldsymbol y^{\mathsf T}  \boldsymbol x)^2 - \sqrt{ab}\,( \boldsymbol y'^{\mathsf T}  \boldsymbol x)^2.
\)
Since \( \boldsymbol x^{\mathsf T} A(J) \boldsymbol x \ge (1 - o(1))\sqrt{ab}\), we have
\(
\sqrt{ab}\bigl[( \boldsymbol y^{\mathsf T}  \boldsymbol x)^2 - ( \boldsymbol y'^{\mathsf T}  \boldsymbol x)^2\bigr] \ge (1 - o(1))\sqrt{ab}.
\)
Cancelling \(\sqrt{ab}\) yields
\begin{equation}
( \boldsymbol y^{\mathsf T}  \boldsymbol x)^2 - ( \boldsymbol y'^{\mathsf T}  \boldsymbol x)^2 \ge 1 - o(1).
\label{eq:*}
\end{equation}
On the other hand, by the unit property,
\begin{equation}
( \boldsymbol y^{\mathsf T}  \boldsymbol x)^2 +( \boldsymbol y'^{\mathsf T}  \boldsymbol x)^2\le 1.
\label{eq:**}
\end{equation}
Comparing (\ref{eq:*}) and (\ref{eq:**}) gives us
\(
1 - o(1) \le( \boldsymbol y^{\mathsf T}  \boldsymbol x)^2 - ( \boldsymbol y'^{\mathsf T}  \boldsymbol x)^2 \le ( \boldsymbol y^{\mathsf T}  \boldsymbol x)^2 + ( \boldsymbol y'^{\mathsf T}  \boldsymbol x)^2 \le 1.
\)
Hence,
\(
(\boldsymbol y'^{\mathsf T} \boldsymbol x)^2 = o(1),\) and \((\boldsymbol y^{\mathsf T} \boldsymbol x)^2 = 1 - o(1).
\)
Therefore,
\(
\boldsymbol y^{\mathsf T} \boldsymbol x = 1 - o(1).
\) Consequently,
\[
\|\boldsymbol x-\boldsymbol y\|_2^2 = \|\boldsymbol x\|^2 + \|\boldsymbol y\|^2 - 2\boldsymbol y^{\mathsf T} \boldsymbol x = 1 + 1 - 2(1 - o(1)) = o(1).
\]
Therefore,
\(
\|\boldsymbol x-\boldsymbol y\|_2 = o(1).
\)
Thus, \eqref{eq:subset-mass} and
\eqref{eq:R-mass} follow from
\[
\sum_{v\in X}x^2_v\leq\sum_{v\in X}|x_v^2-y_v^2|+\sum_{v\in X}y_v^2,\ \  \sum_{v\in X}|x_v^2-y_v^2|
 \le\|x-y\|_2\|x+y\|_2=o(1),
\]
uniformly in $X$.
\end{proof}
\begin{lemma}
\label{lem:local-wheel}
For every $w\in V(H)$, if
$d_B(w)>4\eta^2b$, then
\(
 d_A(w)\le\eta^2a+k-1.
\)
If $a\to\infty$, the symmetric assertion also holds.
Consequently, when $a\to\infty$,
\(
 d_A(u)\le2\eta^2a\) for all  \(u\in A\),   and
 \(d_B(v)\le2\eta^2b\)   for all  \(v\in B.
\)
\end{lemma}
\begin{proof}
Suppose that $d_B(w) > 4\eta^2 b$ and $d_A(w) > \eta^2 a + k - 1$. Since $|A \setminus A^*| = o(a)$ by \eqref{eq:good-sets}, the set $N_A(w) \cap A^*$ contains $k$ distinct vertices $v_1,\dots,v_k$ for all sufficiently large $h$, with indices taken modulo $k$. Each \(v_i\) misses at most \(\eta^2 b\) vertices of \(B\), and hence
\[
|N_B(w) \cap N_B(v_i) \cap N_B(v_{i+1})| \geq d_B(w) - 2\eta^2 b > 2\eta^2 b.
\]
Since \(\eta\) is constant and \(b \to \infty\) by (\ref{eq:ab}), we may choose distinct \(u_i \in N_B(w) \cap N_B(v_i) \cap N_B(v_{i+1})\). Then, \(v_1 u_1 v_2 u_2 \cdots v_k u_k v_1\) forms a \(2k\)-cycle contained in \(H[N_H(w)]\). This gives a copy of \(W_{2k+1}\) centered at \(w\), a contradiction. This proves the first assertion.

If \(a \to \infty\), the same argument with the roles of \(A\) and \(B\) reversed proves the symmetric assertion. Finally, if \(u \in A\), then \(d_B(u) \geq b/2 > 4\eta^2 b\) by (\ref{eq:minimality}); hence \(d_A(u) \leq \eta^2 a + k - 1 \leq 2\eta^2 a\) for large \(h\). The assertion for vertices of \(B\) follows symmetrically.
\end{proof}

\begin{lemma}
\label{lem:R-degree}
There is a constant $C_R=C_R(k)$ such that
\(
 d_H(w)\le C_R
\) for $w\in R$.
\end{lemma}

\begin{proof}
Suppose otherwise.  Then along a sequence of counterexamples there is
$w\in R$ with $d_H(w)\to\infty$.  Since every edge incident with $R$ is
an edit-error, $d_H(w)\le D=o(h)$, and
Lemma~\ref{lem:degree-coordinate} applies.  Cauchy--Schwarz inequality and the
eigen-equation give
\begin{equation}
 \sum_{v\in N_H(w)}x_v^2
 \ge\frac{\rho^2x_w^2}{d_H(w)}
 \ge\frac{1-21\varepsilon}{2}-o(1).
\label{eq:R-lower}
\end{equation}

If $d_B(w)\le4\eta^2b$, then
Lemma~\ref{lem:mass-stability} and~(\ref{eq:minimality}) give
\[
 \sum_{v\in N_H(w)}x_v^2
 \le\frac{d_A(w)}{2a}+\frac{d_B(w)}{2b}+o(1)
 \le\frac14+2\eta^2+o(1),
\]
contradicting~\eqref{eq:R-lower}.

Suppose $d_B(w)>4\eta^2b$.  If $a\ge2k$, then
Lemma~\ref{lem:local-wheel} gives
\[
 \frac{d_A(w)}{2a}
 \le\frac{\eta^2}{2}+\frac{k-1}{2a}
 \le\frac14-\frac1{4k}+\frac{\eta^2}{2}.
\]
Together with $d_B(w)<b/2$ and Lemma \ref{lem:mass-stability}, this gives $\sum_{v\in N_H(w)}x_v^2
 \le1/2-1/(4k)+O(\eta^2)+o(1)$, again contradicting
\eqref{eq:R-lower}.  If $a<2k$, integrality and $d_A(w)<a/2$ give
\[
 \frac{d_A(w)}{2a}\le\frac14-\frac1{8k},
\]
and the same contradiction follows.  The hierarchy
$\varepsilon\ll\eta^4\ll k^{-1}$ completes the proof.
\end{proof}
By Lemma \ref{lem:R-degree}, we obtain \(x_w\leq \frac{C_R}{\rho}x_{u^*}<x_{u^*}\) for large $h$. So that, \(u^*\in A\) or \(B\).
Recall that $A^*=A_{\eta^2}=\{u\in A:d_B(u)\ge(1-\eta^2)b\}$ and $B^*=B_{\eta^2}=\{u\in B:d_A(u)\ge(1-\eta^2)a\}$.
{
\begin{lemma}
\label{lem:4.4}
Suppose that $a \to \infty$. Then
\begin{equation}\label{eq:29}
\begin{split}
x_A^* &\le \bigl(1+O(\eta^2)+o(1)\bigr)\sqrt{\frac{b}{2h}} + \bigl(O(\eta^2)+o(1)\bigr)x^*, \\
x_B^* &\le \bigl(1+O(\eta^2)+o(1)\bigr)\sqrt{\frac{a}{2h}} + \bigl(O(\eta^2)+o(1)\bigr)x^*.
\end{split}
\end{equation}
\end{lemma}
\begin{proof}
Since $a\to\infty$, the assumption $a\le b$ and \(ab=(1+o(1))h\) imply $a+b = o(h)$.
Lemma~\ref{lem:local-wheel} gives $d_A(u)\le 2\eta^2 a$ for $u\in A$ and $d_B(v)\le 2\eta^2 b$ for $v\in B$.

Choose $u_0\in A^*$ to be a vertex attaining the minimum Perron coordinate over $A^*$.
Then $x_{u_0}^2 \le \frac{1}{|A^*|}\sum_{u\in A^*}x_u^2 \le \frac{1}{|A^*|}\sum_{u\in A}x_u^2$.
Recall from (\ref{eq:ab}) that $ab=(1+o(1))h$, and from (\ref{eq:good-sets}) that $|A^*|=(1-o(1))a$.
Combining this with \(\sum_{u\in A}x_u^2\leq 1-\sum_{v\in B^*}x_v^2\leq \frac{1}{2}(1+C\eta^2+o(1))\) (by Lemma \ref{lem:4.2}) yields
\begin{equation}
\label{eq:30}
x_{u_0} \le \bigl(1+O(\eta^2)+o(1)\bigr)\sqrt{\frac{b}{2h}}.
\end{equation}

Choose $u^*\in A$ such that $x_{u^*}=x_A^*$.
Subtracting the eigen-equations at $u^*$ and $u_0$ gives
$\rho(x_{u^*}-x_{u_0}) \le \sum_{w\in N_H(u^*)\setminus N_H(u_0)} x_w$.
Applying the eigen-equation once more yields
\[
\rho^2(x_{u^*}-x_{u_0}) \le \sum_{w\in N_H(u^*)\setminus N_H(u_0)} \rho x_w
\le \sum_{w\in N_H(u^*)\setminus N_H(u_0)} d_H(w)\,x^*.
\]

We now estimate the degree sum $\sum_{w\in N_H(u^*)\setminus N_H(u_0)}d_H(w)$.
Its regular $A$--$B$ incidences are at most $d_A(u^*)b + |N_B(u^*)\setminus N_B(u_0)|\,a$.
Since $u_0\in A^*$, we have $d_B(u_0)\ge (1-\eta^2)b$.
The vertices in $N_B(u^*)\setminus N_B(u_0)$ belong to the set of vertices in $B$ not adjacent to $u_0$, whose size is thus at most $\eta^2 b$.
Meanwhile, Lemma \ref{lem:local-wheel} yields $d_A(u^*)\le 2\eta^2 a$.
Consequently,
\[
d_A(u^*)b + |N_B(u^*)\setminus N_B(u_0)|\,a \le 3\eta^2 ab.
\]
On the other hand, every remaining incidence comes from an edit-error edge, and each such edge is counted at most twice.
Therefore,
\[
\rho^2(x_{u^*}-x_{u_0}) \le (3\eta^2 ab + 2D)\,x^*.
\]
Combining $\rho^2>h$, $ab=(1+o(1))h$, and $D=o(h)$ with \eqref{eq:30}, we obtain the first inequality in \eqref{eq:29}.
The second follows by symmetry.
\end{proof}}
The following lemma is due to Fang, Zhai, and Zhang \cite[Lemmas 4.5 and 4.6]{FangZhaiZhang}.
\begin{lemma}[Fang--Zhai--Zhang \cite{FangZhaiZhang}]
\label{lem:localization}
Suppose $b<100a$. Then
\begin{enumerate}[label=\rm(\roman*)]
\item $x_A^*\le(1+O(\eta^2)+o(1))\sqrt{\frac{b}{2h}},\
 x_B^*\le(1+O(\eta^2)+o(1))\sqrt{\frac{a}{2h}},$
where $x_A^*=\max_{u\in A}x_u$ and $x_B^*=\max_{u\in B}x_u$.
\item \(
 d_B(u)>(1-\eta)b\) for \(u\in A,\) and \(
 d_A(v)>(1-\eta)a\) for \(v\in B\).
\end{enumerate}
\end{lemma}
\begin{lemma}\label{lem3.15}
If  $b\ge100a$ and $u^*$ has maximum Perron coordinate, then $d_B(u^*)\ge(1-3\eta^2)b$ for $u^*\in A,\, x_v\le\frac14x_{u^*}$ for $v\in B,$
and every $u\in A$ with $x_u>\frac{3}{5}x_{u^*}$ satisfies $d_B(u)>11b/20$.
\end{lemma}
\begin{proof}
{Note that $b\geq100a$. We may assume that either $a \to \infty$ or $a$ remains bounded as $h \to \infty$. For the latter, since $k$ is fixed, we may write $a = O_k(1)$.

First we treat the case $a \to \infty$. In this case, there exists a $u^* \in A$, i.e., $x^* = x_A^*$. Indeed, suppose $x^* = x_B^*$. Upon absorbing the $O(\eta^2)x^*$ term, the second inequality in (\ref{eq:29}) yields
\begin{equation*}
x^* \leq (1 + O(\eta^2)) \sqrt{\frac{a}{2h}}.
\end{equation*}
Note that the first assertion of Lemma \ref{lem:4.2} guarantees that $x_u \geq (1 - O(\eta^2))\sqrt{\frac{b}{2h}}$ for all $u \in A^*$. Since $b \geq 100a$, $x_u > x^*$ for all $u \in A^*$, a contradiction. Hence, $u^* \in A$, and the same lower bound for $u \in A^*$ together with the first inequality in (\ref{eq:29}) gives
\begin{equation}
\label{eq:33}
x_{u^*} = (1 + O(\eta^2) + o(1)) \sqrt{\frac{b}{2h}}.
\end{equation}

Now, we show $d_B(u^*) \geq (1 - 3\eta^2)b$. Suppose $d_B(u^*) < (1 - 3\eta^2)b$. Since $u^* \in A$, Lemma \ref{lem:local-wheel} gives $d_A(u^*) \leq 2\eta^2a$. Together with \(
\rho^2 x_w \leq \sum_{u \in N_A(w)} d_B(u) x_B^* + \sum_{u \in N_B(w)} d_A(u) x_A^* + 2D x^* \leq d_A(w) b x_B^* + d_B(w) a x_A^* + o(h) x^*\), one has
\begin{equation*}
\rho^2 x_{u^*} \leq (d_A(u^*)b + d_B(u^*)a + o(h)) x_{u^*} \leq ((1 - \eta^2)ab + o(h)) x_{u^*}.
\end{equation*}
Recall that $ab = (1 + o(1))h$ and $\eta^2$ is a constant. Thus we have $\rho^2 <h$, which contradicts the lower bound with \(\rho> \sqrt{h}\). Consequently, $d_B(u^*) \geq (1 - 3\eta^2)b$.

Next, since $\sqrt{b/a} \geq 10$, the second inequality in (\ref{eq:29}) together with (\ref{eq:33}) directly implies $x_B^* \leq \frac{1}{4}x_{u^*}$, i.e., $x_v \leq \frac{1}{4}x_{u^*}$ for any vertex $v \in B$. Now take $u \in A$ with $x_u > \frac{3}{5}x_{u^*}$, and set $\beta_u = d_B(u)/b$. Again by Lemma \ref{lem:local-wheel}, we have $d_A(u) \leq 2\eta^2a$. Together with \(
\rho^2 x_w \leq \sum_{u \in N_A(w)} d_B(u) x_B^* + \sum_{u \in N_B(w)} d_A(u) x_A^* + 2D x^* \leq d_A(w) b x_B^* + d_B(w) a x_A^* + o(h) x^*\) and $x_B^* \leq \frac{1}{4}x_{u^*}$, we obtain
\begin{equation*}
\rho^2 x_u \leq d_A(u)b x_B^* + (d_B(u)a + o(h)) x_{u^*} \leq \left(ab\left(\frac{1}{2}\eta^2 + \beta_u\right) + o(h)\right) x_{u^*}.
\end{equation*}
Recall that $ab = (1 + o(1))h$ and $\rho^2 > h$. Thus, $x_u \leq (\frac{1}{2}\eta^2 + \beta_u + o(1))x_{u^*}$. Combining this with $x_u > \frac{3}{5}x_{u^*}$, we obtain $\beta_u > 11/20$ for sufficiently small $\eta$ and large $h$.

Second we treat the case $a = O_k(1)$. Now, $b = \Theta(h)$ and $D = o(h) = o(ab)$. Since $\eta^2$ is a constant, and the total number of missing $A$--$B$ edges is at most $D$, we have
\begin{equation*}
d_B(u) \geq b - D = (1 - o(1))b \geq (1 - 3\eta^2)b > \frac{11b}{20}
\end{equation*}
for all vertices $u \in A$. 

We next verify that \( u^* \in A \) and \( x_v = o(1) \) for every \( v \in B \). Let \( K \) be the union of \( K_{A,B} \) and \( |R| \) isolated vertices. Denote \( M = A(H) - A(K) \), and let \( \rho(M) \) denote the maximum absolute value of the eigenvalues of \( M \). Each of the \( D \) edge discrepancies between \( G \) and \( K \) contributes two nonzero off-diagonal entries, each is equal to \(1\) or \(-1\), to \( M \). Hence, \( \operatorname{tr}(M^2) = 2D \), and thus \( \rho(M) \leq \sqrt{2D} \). Note that \( \rho(K) = \sqrt{ab} \). By Weyl's inequality, \( |\rho(H) - \sqrt{ab}| \leq \sqrt{2D} = o(\sqrt{ab}) \), which yields
\begin{equation}
\label{eq:34}
\rho(H) \geq (1 - o(1))\sqrt{ab}.
\end{equation}

Let \( p = \sum_{u \in A} x_u^2 \) and \( q = \sum_{v \in B} x_v^2 \). Then, \( p + q \leq1 \). Since \(\boldsymbol{x}^{\mathsf T} M\boldsymbol{x} \leq \rho(M) \leq \sqrt{2D} \), we have
\[
\rho(H) = \boldsymbol{x}^{\mathsf T} A(K)\boldsymbol{x} + \boldsymbol{x}^{\mathsf T} M\boldsymbol{x} \leq 2\left(\sum_{u \in A} x_u\right)\left(\sum_{v \in B} x_v\right) + \sqrt{2D}.
\]
Applying the Cauchy--Schwarz inequality on \( A \) and \( B \) yields
\[
\rho(H) \leq 2\sqrt{ap}\sqrt{bq} + \sqrt{2D} = (2\sqrt{pq} + o(1))\sqrt{ab}.
\]
Together with (\ref{eq:34}), this gives \( 1 - o(1) \leq 2\sqrt{pq} \). Combining this with \( 2\sqrt{pq} \leq p + q \leq 1 \) implies \( p, q = \frac{1}{2} + o(1) \). Together with \( a = O_k(1) \), this implies
\[
x_A^* \geq \sqrt{\frac{1}{a}\sum_{u \in A} x_u^2} = \sqrt{\frac{p}{a}} \geq J_k
\]
for some positive constant \( J_k \) that depends only on \( k \).

Let \( v \in B \). Every neighbor of \( v \) outside \( A \) is incident to an edge in the edit-error set between \( H \) and \( K \). Hence, \( |N_H(v) \setminus A| \leq D \). By the eigen-equation and the Cauchy--Schwarz inequality,
\[
\rho x_v \leq \sum_{u \in A} x_u + \sum_{u \in N_H(v) \setminus A} x_u \leq \sqrt{ap} + \sqrt{D}\sqrt{\sum_{u \in N_H(v) \setminus A} x_u^2} \leq \sqrt{ap} + \sqrt{D}.
\]
Recall that \( a = O_k(1) \), \( \sqrt{D} = o(\sqrt{ab}) \) and \( \rho > \sqrt{h} = \Theta(\sqrt{ab}) \). Thus \( x_v = o(1) \). Together with \( x_{u^*} \geq x_A^* \geq J_k \), this implies \( u^* \in A \) and \( x_v = o(x_{u^*}) \) for all \( v \in B \). The proof is therefore complete.}
\end{proof}
\begin{lemma}
\label{lem:absorb-R}
\begin{enumerate}[label=\rm(\roman*)]
\item If  $b<100a$, i.e., \(a\to\infty\), then the number of edges with at least one endpoint in \(R\) is at most \(k-1\). Moreover, the vertices incident with these edges contribute \( o(1) \) to the Rayleigh quotient.
\item If  there exists an \(A_0\) such that $a<A_0$, then $|R|=o(b)$; after moving $R$ into $B$, every moved vertex has coordinate $o(x_{u^*})$, and all conclusions of Lemma~\ref{lem3.15} remain valid.
\end{enumerate}
\end{lemma}
\begin{proof}
By Lemma~\ref{lem:R-degree}, $d(w)\le C_R$ for $w\in R$ and
\(
 x_w\le\frac{C_Rx_{u^*}}{\rho}.
\)
When $b<100a$, Lemma~\ref{lem:localization} gives
$x^2_{u^*}=O(1/a)$ (when \(a=O_k(1)\)) and $x^2_{u^*}=O(h^{-1/2})$ (when \(a\to \infty\)).  Assume $a$ is sufficiently large.  Then every edge
incident with $R$ satisfies~\eqref{eq:light}, and
Lemma~\ref{lem:light-edges} shows that there are at most $k-1$ such
edges.  Since there are no isolated vertices, $|R|=O_k(1)$.

If there exists an \(A_0\) such that $a<A_0$, then $b=\Theta(h)$.  Every vertex of $R$ is incident with an edit-error edge, so $|R|\le2D=o(h)=o(b)$.  Moving $R$ into $B$ changes $b$ by $o(b)$, and $x_w\le C_Rx_{u^*}/\rho=o(x_{u^*})$. Thus the missing-cross-edge and localization estimates are unchanged up to $o(1)$. Repeating the argument of Lemma \ref{lem:localization}, we obtain the same assertion.
\end{proof}
\subsection{The bounded-core reduction}

We first exclude the balanced scale.
\begin{lemma}
\label{lem:no-balanced}
For all sufficiently large $h$, the stability partition $A\cup B$ with $|A|=a,|B|=b$ satisfies
\(
 b\ge100a.
\)
\end{lemma}
\begin{proof}
Suppose $a\le b<100a$.  Then $a,b=\Theta(\sqrt h)$. By
Lemma~\ref{lem:localization}(ii),  one has
\(
 d_B(u)>(1-\eta)b\) for \(u\in A,\) and \(
 d_A(v)>(1-\eta)a\) for \(v\in B\).
Furthermore, we claim
\[
 \Delta(H[A]),\Delta(H[B])\le k-1.
\]
In fact, suppose first that $\Delta(H[A]) \geq k$. Then there exists a vertex $u\in A$ having $k$ distinct neighbors $u_1,u_2,\dots,u_k\in A$. Read the indices cyclically, so that $u_{k+1}=u_1$. For every $i\in\{1,\dots,k\}$, it follows from Lemma~\ref{lem:localization}(ii) that
\[
\big|N_B(u)\cap N_B(u_i)\cap N_B(u_{i+1})\big|
\geq d_B(u)+d_B(u_i)+d_B(u_{i+1})-2b > (1-3\eta)b.
\]
Since $k$ is fixed and $b\to\infty$, each of these $k$ triple intersections has cardinality tending to infinity. Thus, for all sufficiently large $h$, we may choose distinct vertices
\[
z_i\in N_B(u)\cap N_B(u_i)\cap N_B(u_{i+1})
\]
for each $i\in\{1,\dots,k\}$. Then, $u_1z_1u_2z_2\cdots u_kz_ku_1$ forms a cycle $C_{2k}$ in $H[N_H(u)]$. Hence, $G$ contains a copy of $W_{2k+1}$, contradicting the assumption that $G$ is $W_{2k+1}$-free. Therefore, $\Delta(H[A])\leq k-1$.
By symmetry, we can prove $\Delta(H[B])\leq k-1$.

Let
\[
 \alpha=\frac{2e(A)}a,\qquad
 \beta=\frac{2e(B)}b,\qquad t=\sqrt{b/a}.
\]
Thus $\alpha,\beta\le k-1$. Let $s = e(A)+e(B)$. Edges incident with the original $R$
contribute $o(1)$ by Lemma~\ref{lem:absorb-R}. So we have
\[
\rho=\sum_{uv\in E(A,B)}2x_ux_v+\sum_{uv\in E(A)}2x_ux_v+\sum_{uv\in E(B)}2x_ux_v+o(1).
\]
For the cross contribution, $e(A,B)\leq h-s$. By the Cauchy--Schwarz inequality,
\[
\biggl(\sum_{uv\in E(A,B)}x_ux_v\biggr)^2
\le (h-s)\sum_{uv\in E(A,B)}x_u^2x_v^2
\le (h-s)\biggl(\sum_{u\in A}x_u^2\biggr)\biggl(\sum_{v\in B}x_v^2\biggr).
\]
Since $\sum_{u\in A}x_u^2+\sum_{v\in B}x_v^2\leq1$, we have $\sum_{u\in A}x_u^2\sum_{v\in B}x_v^2\le \frac14$, and therefore,
\(
\sum_{uv\in E(A,B)}2x_ux_v\le \sqrt{h-s}.
\)

Note that where $ab=(1+o(1))h$ and $e(A)/a=\alpha/2\le (k-1)/2$. Then in view of Lemma~\ref{lem:localization}(i), we obtain
\[
\sum_{uv\in E(A)}2x_ux_v
\le 2e(A)(x_A^*)^2
\le e(A)\frac{b}{h}\bigl(1+O(\eta^2)+o(1)\bigr)
=\frac{e(A)}{a}+O_k(\eta^2)+o(1),
\]
By symmetry, we also get
\[
\sum_{uv\in E(B)}2x_ux_v\le \frac{e(B)}{b}+O_k(\eta^2)+o(1).
\]
Together with \(\sum_{uv\in E(A,B)}2x_ux_v\le \sqrt{h-s}\) and the above two inequalities, we obtain
\[
\rho=\sum_{uv\in E(H)}2x_ux_v
\le \sqrt{h-s}+\frac{e(A)}{a}+\frac{e(B)}{b}+O_k(\eta^2)+o(1).
\]

Recall that $a\le b<100a$. Thus $1\le t=\sqrt{b/a}<10$. Furthermore, $ab=(1+o(1))h$ implies $a=(1+o(1))\sqrt{h}/r$ and $b=(1+o(1))r\sqrt{h}$. Since $e(A)=\alpha a/2$, $e(B)=\beta b/2$ and $\alpha,\beta\le k-1$, it follows that $s=(\alpha a+\beta b)/2=O_k(\sqrt{h})$. Therefore,
\[
\sqrt{h-s}=\sqrt{h}-\frac{s}{2\sqrt{h}}+o(1)
=\sqrt{h}-\frac{\alpha}{4t}-\frac{\beta t}{4}+o(1).
\]
Moreover, $e(A)/a=\alpha/2$ and $e(B)/b=\beta/2$. So we have
\[
\rho
\le \sqrt{h}-\frac{\alpha}{4r}-\frac{\beta r}{4}+\frac{\alpha}{2}+\frac{\beta}{2}+O_k(\eta^2)+o(1)
=\sqrt h+
 \alpha\left(\frac12-\frac1{4t}\right)+
 \beta\left(\frac12-\frac t4\right)+
 O_k(\eta^2)+o(1).\]

If $\min\{\alpha,\beta\}\le k-1-1/4$, elementary maximization over
$1\le t<10$ gives
\[
 \rho\le\sqrt h+\frac{k-1}{2}-\gamma(k,t)+O_k(\eta^2)+o(1)
\]
for some $\gamma(k,t)>0$, after choosing $\eta$ sufficiently small.  This
contradicts
\[
 \rho\ge s_k(h)=\sqrt h+\frac{k-1}{2}+O_k(h^{-1/2}).
\]
Hence $\alpha,\beta>k-1-1/4$.  Since the internal maximum degrees are
at most $k-1$, there $u_0\in A$ so that $d_A(u_0)=k-1$; write its internal
neighbors as $u_1,\ldots,u_{k-1}$.  By localization, their common
neighborhood
\[
 X=\bigcap_{i=0}^{k-1}N_B(u_i)
\]
By Lemma \ref{lem:localization}(ii), \( |B \setminus N_B(u_i)| < \eta b \) for \( i \in \{0, \ldots, k-1\} \). Since any \( v \in B \setminus X \) lies in \( B \setminus N_B(u_i) \) for some index \( i \), we have \( B \setminus X \subseteq \bigcup_{i=0}^{k-1} (B \setminus N_B(u_i)) \). Thus
\[
|B \setminus X| \leq \sum_{i=0}^{k-1} |B \setminus N_B(u_i)| < k\eta b,
\]
and so
\[
|X| = b - |B \setminus X| > (1-k\eta)b.
\] Let \( E_{\text{out}} \) be the set of edges of \( H[B] \) having at least one endpoint within \( B \setminus X \). Recall \( \Delta(H[B]) \leq k-1 \). Therefore,
\[
|E_{\text{out}}| \leq \sum_{v \in B \setminus X} d_B(v) \leq (k-1)|B \setminus X| < k(k-1)\eta b.
\]
Recall \( \beta > k-1 - \frac{1}{4} \). Thus \( e(B) = \frac{\beta b}{2} > \left(\frac{k-1}{2} - \frac{1}{8}\right)b \). Recall that \( \eta \ll 1 \) and \( b \) is sufficiently large. It follows that
\[
e(X) = e(B) - |E_{\text{out}}| > \left(\frac{k-1}{2} - \frac{1}{8} - k(k-1)\eta\right)b > 2(k-1).
\]

We next show that \( H[X] \) contains at least two disjoint edges. Indeed, suppose otherwise, and let \( pq \in E(X) \). Since there are no two disjoint edges in \( H[X] \), every edge of \( H[X] \) must be incident with at least one of \( p \) and \( q \). Therefore,
\[
e(X) \leq d_X(p) + d_X(q) \leq 2(k-1),
\]
a contradiction.

In summary $H[X]$ has more than $2(k-1)$ edges
and hence contains two disjoint edges, say \(p_1q_1\) and \(p_2q_2\).

When $k=3$, $u_1 p_1 q_1 u_2 p_2 q_2 u_1$ forms a $C_6$. Now suppose $k\ge 4$. Since $|X| > (1-k\eta)b$, we may select distinct vertices $z_2,\dots,z_{k-2}\in X\setminus\{p_1,q_1,p_2,q_2\}$. The walk
\[
u_1 p_1 q_1 u_2 z_2 u_3 z_3 \cdots u_{k-2} z_{k-2} u_{k-1} p_2 q_2 u_1
\]
is a cycle $C_{2k}$. Indeed, the two pairs $p_1 q_1$ and $p_2 q_2$ are edges of $H[X]$, while every remaining edge in the exhibited cycle connects some $u_i$ to a vertex in $X$. Such edges exist by the definition of $X$. Recall that $u_1,\dots,u_{k-1}\in N_A(u_0)$ and $X\subseteq N_B(u_0)$. We thus obtain a copy of $C_{2k}$ with all vertices contained in $N(u_0)$. Consequently, $H$ contains a copy of $W_{2k+1}$, contradicting the assumption that $G$ is $W_{2k+1}$-free. Therefore, the case $a\le b<100a$ cannot occur.
\end{proof}

We now work in the split case $b\ge100a$.  Take a maximum Perron vertex
$u^*$ and put
\[
 L=\left\{u\in N_A(u^*):x_u>\frac{3x_{u^*}}{5}\right\},\qquad
 L^*=N(u^*)\setminus L,\qquad
 W=V(H)\setminus N[u^*].
\]
By Lemmas~\ref{lem:localization} and~\ref{lem:absorb-R},
\begin{equation}
 u^*\in A,\quad d_B(u^*)\ge(1-o(1))b,\quad
 x_v\le \frac{x_{u^*}}{4}\quad(v\in B),
\label{eq:split-localization}
\end{equation}
and every $u\in L$ has $d_B(u)>11b/20$.

\begin{lemma}
\label{lem:L-size}
$|L|\le k-1$.
\end{lemma}

\begin{proof}
If $v_1,\ldots,v_k$ are distinct vertices of $L$,  Lemma \ref{lem:localization} gives \(d_B(v_i)>\frac{11}{20}b\), while \(d_B(u^*)\geq (1-3\eta^2)b\). So that, for every $i$
the triple intersection
\[
 N_B(u^*)\cap N_B(v_i)\cap N_B(v_{i+1})\geq d_B(u^*)+d_B(v_i)+d_B(v_{i+1})-2b>(\frac{1}{10}-3\eta^2)b.
\]
For sufficiently small \(\eta\) and sufficiently large \(h\), the right-hand side is larger than \(k\). We may therefore choose pairwise distinct vertices \(w_i \in N_B(u^*) \cap N_B(v_i) \cap N_B(v_{i+1})\) for each \(i \in \{1, \ldots, k\}\). Thus, \(v_1 w_1 v_2 w_2 \cdots v_k w_k v_1\) is a copy of \(C_{2k}\) in \(H[N_H(u^*)]\). Together with the center \(u^*\), this gives a copy of \(W_{2k+1}\), a contradiction. Hence, \(|L| \leq k-1\).
\end{proof}

Put
\[
 \cP:=\rho^2-(k-1)\rho-h+c_k.
\]
Theorem~\ref{thm:FZZ} gives $\cP\le0$, while
Lemmas~\ref{lem:candidate-expansion} and the monotonicity of
$x^2-(k-1)x$ give
\begin{equation}
 -c_k\le\cP\le0.
\label{eq:P-bounded}
\end{equation}
Define
\[
 S_{L^*}:=
 \sum_{yz\in E(L^*)}(x_y+x_z-x_{u^*}).
\]

\begin{lemma}
\label{lem:S-defect}
There is a constant $C>0$ such that
\begin{equation}
 S_{L^*}\le
 -Cb|L^*\cap A|x_{u^*}-\frac12e(L^*\cap B)x_{u^*}\le0
\label{eq:S-defect}
\end{equation}
for all sufficiently large $h$.
\end{lemma}

\begin{proof}
By the definition of $L$, $x_u\le \frac{3}{5}x_{u^*}$ for every $u\in L^*\cap A$, and Lemma \ref{lem:localization}(iii) gives $x_v\le \frac14 x_{u^*}$ for every $v\in L^*\cap B$. Consequently, an edge inside $L^*\cap A$ contributes at most $\frac15 x_{u^*}$ to $S_{L^*}$, an edge between $L^*\cap A$ and $L^*\cap B$ contributes at most $-\frac{3}{20}x_{u^*}$, and an edge inside $L^*\cap B$ contributes at most $-\frac12 x_{u^*}$. On the other hand, every $u\in L^*\cap A$ has at least
$(1/2-o(1))b$ neighbors in
$N_B(u^*)=L^*\cap B$, whereas
$e(L^*\cap A)\le a|L^*\cap A|/2$.  Since $b\ge100a$, the negative
cross contribution dominates and gives~\eqref{eq:S-defect}.
\end{proof}

\begin{lemma}
\label{lem:split-defect}
It holds that $|L|=k-1$ and
\begin{align}
 -S_{L^*}
 &+\sum_{v\in L}(\rho+d_{L^*}(v))(x_{u^*}-x_v)+\sum_{w\in W}\bigl(d_L(w)x_w+f(w)\bigr)\notag\\
 &=(c_{k-1}-e(L)-\cP)x_{u^*}=O_k(x_{u^*}).
\label{eq:master-defect}
\end{align}
\end{lemma}

\begin{proof}
For each $u\in L$, the eigen--equation at $u$ reads
\[
\rho x_u = \sum_{v\in N_H(u)} x_v.
\]
We decompose $N_H(u)$ into four parts: $N_H(u)\cap L$, $N_H(u)\cap L^*$, $N_H(u)\cap W$, and the vertex $u^*$. Summing over all $u\in L$ gives
\begin{align}
\label{eq:30}
\rho \sum_{u \in L} x_u
&= \sum_{u \in L} \sum_{v \in N_H(u) \cap L} x_v
+ \sum_{u \in L} \sum_{v \in N_H(u) \cap L^*} x_v
+ \sum_{u \in L} \sum_{v \in N_H(u) \cap W} x_v
+ |L| x_{u^*} \notag \\
&= \sum_{uv \in E(L)} (x_u + x_v)
+ \sum_{uv \in E(L, L^*)} x_v
+ \sum_{w \in W} d_L(w) x_w
+ |L| x_{u^*}.
\end{align}
Recall the residual identity from Lemma \ref{lem:residual} with \(d = k-1\),  \(c = c_k\) and (\ref{eq:30}). Then we have
\begin{align*}
 \cP x_{u^*}
 ={}&S_{L^*}-(k-1-|L|)\rho x_{u^*}+
 (c_k-|L|-e(L))x_{u^*}\\
 &-\sum_{v\in L}(\rho+d_{L^*}(v))(x_{u^*}-x_v)-\sum_{w\in W}\bigl(d_L(w)x_w+f(w)\bigr).
\end{align*}
All displayed defect terms are nonnegative after moving them to the
left, and $S_{L^*}\le0$.  If $|L|\le k-2$, this identity and
\eqref{eq:P-bounded} imply
$\cP\le-\rho+O_k(1)$, a contradiction.  Thus $|L|=k-1$, and the
identity simplifies to~\eqref{eq:master-defect}.
\end{proof}

\begin{lemma}
\label{lem:bounded-core}
There is a set $K\subseteq V(H)$ with $|K|=O_k(1)$ such that
$K':=V(H)\setminus K$ is an independent set with $N_H(K')\subseteq K$.
\end{lemma}

\begin{proof}
By \eqref{eq:S-defect} and~\eqref{eq:master-defect} we obtain
\(
 L^*\cap A=\varnothing,\) and \(e(L^*)=O_k(1).
\)
Moreover, the $d_W(w)x_{u^*}/2$ term in $f(w)$ gives
$e(W)=O_k(1)$.

For $w\in W\cap B$,~\eqref{eq:split-localization} implies
\(
 f(w)\ge\frac34d_{N(u^*)}(w)x_{u^*}+\frac12d_W(w)x_{u^*}.
\)
Every such nonisolated vertex therefore costs at least $x_{u^*}/2$ in
\eqref{eq:master-defect}; hence $|W\cap B|=O_k(1)$.

Now let $w\in W\cap A$.  By~\eqref{eq:minimality}, after the harmless
absorption in Lemma~\ref{lem:absorb-R},
\(
 d_{L^*}(w)\ge b/3.
\)
The term $f(w)$ in~\eqref{eq:master-defect} then gives
\(
 x_{u^*}-x_w=O_k(x_{u^*}/b),
\)
so $x_w\ge x_{u^*}/2$ for large $h$.  Consequently vertices with
$d_L(w)>0$ are bounded in number by the $d_L(w)x_w$ term.
The endpoints of edges in $H[W]$ are also bounded in number because
$e(W)=O_k(1)$.

For every remaining $w\in W\cap A$, we have
$d_L(w)=d_W(w)=0$. Hence $N(w)\subseteq L^*$.  Comparing the
eigen-equations at $u^*$ and $w$ gives
\[
 \rho(x_{u^*}-x_w)
 \ge\sum_{v\in L}x_v=(k-1-o(1))x_{u^*};
\]
here~\eqref{eq:master-defect} was used to deduce
$x_v=x_{u^*}-O_k(x_{u^*}/\rho)$ for $v\in L$.  Since
$d_{L^*}(w)\ge b/3$ and, by~\eqref{eq:ab},
$b/\rho\ge9$ in the split case, each such $w$ contributes at least
$c_k'x_{u^*}$ to $f(w)$.  Thus $|W|=O_k(1)$.

Let $Q$ be the set of endpoints of edges in $H[L^*]$, and put
\[
 K=\{u^*\}\cup L\cup W\cup Q,
 \qquad K'=L^*\setminus Q.
\]
Then $|K|=O_k(1)$, $K'$ is an independent set, and all neighbors of $K'$ lie
in $K$.
\end{proof}

\subsection{Finite-dimensional closure}
Let $K,K'$ be defined as in Lemma~\ref{lem:bounded-core}.  Partition $K'$ by
neighborhood type in $K$.  Write the distinct nonempty types as
$T_1,\ldots,T_s\subseteq K$, their multiplicities as $n_1,\ldots,n_s$,
and their incidence vectors as
$\chi_1,\ldots,\chi_s\in\{0,1\}^{K}$.  Set
\[
 M=\sum_{i=1}^s n_i\chi_i\chi_i^{\mathsf T}.
\]
Let $\boldsymbol{y}$ be the restriction of a Perron vector to $K$, and put
$\tilde{\boldsymbol{y}}=\boldsymbol{y}/\|\boldsymbol{y}\|_2$. We are to obtain the following finite-core identity.
\begin{lemma}
\label{lem:finite-identity}
It holds
\begin{equation}
 \rho^2\boldsymbol{y}=\rho A(K)\boldsymbol{y}+M\boldsymbol{y}.
\label{eq:finite-vector}
\end{equation}
Let
\(
 \xi:=\tr M-\tilde{\boldsymbol{y}}^{\mathsf T}M\tilde{\boldsymbol{y}}.
\)
Then $\xi\ge0$ and
\begin{equation}
 \cP=-\xi+
 \rho\bigl(\tilde{\boldsymbol{y}}^{\mathsf T}A(K)\tilde{\boldsymbol{y}}-(k-1)\bigr)+c_k-e(K).
\label{eq:finite-residual}
\end{equation}
\end{lemma}

\begin{proof}
All vertices of type $T_i$ have the same Perron coordinate
$\chi_i^{\mathsf T}\boldsymbol{y}/\rho$.  Substitution in the eigen-equations on
$K$ gives~\eqref{eq:finite-vector}.  Multiplying it by
$\boldsymbol{y}^{\mathsf T}/\|\boldsymbol{y}\|_2^2$  yields
\(
 \rho^2=\rho\tilde{\boldsymbol{y}}^{\mathsf T}A(K)\tilde{\boldsymbol{y}}+\tilde{\boldsymbol{y}}^{\mathsf T}M\tilde{\boldsymbol{y}}.
\)
Since
\[
 h=e(K)+\sum_i n_i|T_i|=e(K)+\tr M,
\]
by the definition of \(\cP\), equation~\eqref{eq:finite-residual} follows.  Finally,
\[
 \xi=\sum_i n_i
 \bigl(\|\chi_i\|_2^2-(v^{\mathsf T}\chi_i)^2\bigr)\ge0
\]
follows by the Cauchy--Schwarz inequality.
\end{proof}

\begin{lemma}
\label{lem:one-type}
After moving $O_k(1)$ vertices from $K'$ into $K$, all vertices of $K'$
have one common neighborhood $T\subseteq K$.  Moreover,
\(
 \Delta(K[T])\le k-1.
\)
\end{lemma}
\begin{proof}
Because $|K|=O_k(1)$ and
$\sum_i n_i|T_i|=h-O_k(1)$, some type $T$ has multiplicity
$n_T=\Theta_k(h)$.  Let $p=|T|$ and $\tilde{\boldsymbol{x}}=\chi_T/\sqrt p$.

If a vertex $z\in T$ had $k$ neighbors
$z_1,\ldots,z_k$ in $K[T]$, choose distinct vertices
$w_1,\ldots,w_k$ from the large twin class.  Then
\[
 z_1w_1z_2w_2\cdots z_kw_kz_1
\]
is a $C_{2k}$ in $N(z)$, a contradiction.  Hence
$\Delta(K[T])\le k-1$, and
\begin{equation}
\tilde{\boldsymbol{x}}^{\mathsf T}A(K)\tilde{\boldsymbol{x}}=\frac{2e(K[T])}{p}\le k-1.
\label{eq:uAu}
\end{equation}

Since $\cP\ge-c_k$, equation~\eqref{eq:finite-residual} first gives
$\xi=O_k(\sqrt h)$.  The contribution of the large type \(T:M_T=n_T\chi_T\chi^{\mathsf T}_T\)  to $\xi$ is
\(
 n_Tp\bigl(1-(\tilde{\boldsymbol{y}}^{\mathsf T}\tilde{\boldsymbol{x}})^2\bigr),
\)
so we have
\[
1-\tilde{\boldsymbol{y}}^{\mathsf T}\tilde{\boldsymbol{x}}\leq1-(\tilde{\boldsymbol{y}}^{\mathsf T}\tilde{\boldsymbol{x}})^2\leq \frac{\xi}{ n_Tp},
\]
and therefore
\[
    \|\tilde{\boldsymbol{y}}-\tilde{\boldsymbol{x}}\|_2=(\tilde{\boldsymbol{y}}-\tilde{\boldsymbol{x}})^{\mathsf T}(\tilde{\boldsymbol{y}}-\tilde{\boldsymbol{x}})=2(1-\tilde{\boldsymbol{y}}^{\mathsf T}\tilde{\boldsymbol{x}})=O_k\bigl(\sqrt{\xi/h}\bigr).
\]
As $h\to\infty$, one has $\|\tilde{\boldsymbol{y}}-\tilde{\boldsymbol{x}}\|_2\to 0$. i.e., \(v=u+o(1)\). Let \(\tilde{\boldsymbol{y}}=\tilde{\boldsymbol{x}}+\boldsymbol{z}\), we have \(||\boldsymbol{z}||_2=\|\tilde{\boldsymbol{y}}-\tilde{\boldsymbol{x}}\|_2=O_k\bigl(\sqrt{\xi/h}\bigr)\).

From the above, substituting~\eqref{eq:uAu} into~\eqref{eq:finite-residual} yields
\[
 -c_k\leq\cP\le-\xi+c_k\sqrt\xi+c_k.
\]
It follows that $\xi=O_k(1)$.

For a different type $T_i\ne T$, its individual defect tends to
\(
 |T_i|-\frac{|T_i\cap T|^2}{p}>0.
\)
The contribution of $T_i$ to $\xi$ is \(\xi_i= n_i
 \bigl(\|\chi_i\|_2^2-(\tilde{\boldsymbol{y}}^{\mathsf T}\chi_i)^2\bigr)= n_i
 \bigl(\|\chi_i\|_2^2-((\tilde{\boldsymbol{x}}+o(1))^{\mathsf T}\chi_i)^2\bigr)\), where \(\|\chi_i\|_2^2=T_i\), \(\tilde{\boldsymbol{x}}^{\mathsf T} \chi_i = \frac{\boldsymbol{x}_T^{\mathsf T}}{\sqrt{p}} \chi_i = \frac{|T_i \cap T|}{\sqrt{p}}.\) Combining this with  $\xi_i\leq\xi=O_k(1)$ yields $n_i=O_k(1)$.  Absorb all these exceptional vertices into $K$.
\end{proof}

We now assume that $|K'|=n$. Let $T$ be the common neighborhood of vertices in $K'$ satisfying $T\subseteq K$ and $|T|=p$.  Let $\chi$ be the incidence vector of $T$. Then we have the following resolvent expansion.
\begin{lemma}
\label{lem:resolvent}
It holds that
\begin{align}
 \rho={}&\sqrt h+\frac{e(K[T])}{p}+\frac{1}{2\sqrt h}
 \left(
 \frac{\sum_{z\in K}d_T(z)^2}{p}-e(K)
 -3\left(\frac{e(K[T])}{p}\right)^2
 \right)+O_k(h^{-1}).
\label{eq:resolvent-expansion}
\end{align}
\end{lemma}
\begin{proof}
By Lemma \ref{lem:bounded-core} and Lemma \ref{lem:one-type}, we have
\(
V(H)=K\cup K',\quad K\cap K'=\emptyset,\quad |K'|=n,
\)
and
\[
A=
\begin{pmatrix}
A(K) & \chi \boldsymbol{1}_n^{\mathsf T} \\
\boldsymbol{1}_n \chi^{\mathsf T} & O_{n\times n}
\end{pmatrix}.
\]
It follows readily that
\begin{equation}
 \rho=n\chi^{\mathsf T}(\rho I-A(K))^{-1}\chi,
\label{eq:schur}
\end{equation}
here $I$ denotes the identity matrix.

For large $\rho$, one has
\(
\rho I - A(K) = \rho\bigl(I - \rho^{-1}A(K)\bigr),
\)
and so
\[
\bigl(\rho I - A(K)\bigr)^{-1} = \frac{1}{\rho}\bigl(I - \rho^{-1}A(K)\bigr)^{-1}.
\]
Provided the operator norm $\bigl\|\rho^{-1}A(K)\bigr\| < 1$, which holds for sufficiently large $\rho$, we expand the inverse via the Von Neumann series:
\[
\bigl(I - \rho^{-1}A(K)\bigr)^{-1}
=\sum_{t=0}^{\infty}\frac{A(K)^t}{\rho^t}
= I+\frac{A(K)}{\rho}+\frac{A(K)^2}{\rho^2}+O_k(\rho^{-3}).
\]
Substituting back yields the resolvent expansion
\[
\bigl(\rho I - A(K)\bigr)^{-1}
=\frac{I}{\rho}+\frac{A(K)}{\rho^2}+\frac{A(K)^2}{\rho^3}+O_k(\rho^{-4}).
\]
Since the dimension of $K$ is a constant depending only on $k$, we write the remainder term as $O_k(\cdot)$.
Multiplying by $\chi^{\mathsf T}$ on the left and $\chi$ on the right gives
\begin{align*}
\chi^{\mathsf T}\bigl(\rho I-A(K)\bigr)^{-1}\chi
&=\chi^{\mathsf T}\left(\frac{I}{\rho}+\frac{A(K)}{\rho^2}+\frac{A(K)^2}{\rho^3}+O_k(\rho^{-4})\right)\chi\\
&=\frac{\chi^{\mathsf T} I\chi}{\rho}+\frac{\chi^{\mathsf T} A(K)\chi}{\rho^2}+\frac{\chi^{\mathsf T} A(K)^2\chi}{\rho^3}+O_k(\rho^{-4}).
\end{align*}
So we have
\[
 \chi^{\mathsf T}(\rho I-A(K))^{-1}\chi
 =\frac{p}{\rho}+
 \frac{2e(K[T])}{\rho^2}+
 \frac{\sum_{z\in K}d_T(z)^2}{\rho^3}
 +O_k(\rho^{-4}).
\]
Furthermore, we obtain
\begin{equation}
 \rho=  \frac{p}{\rho}+
 \frac{2e(K[T])}{\rho^2}+
 \frac{\sum_{z\in K}d_T(z)^2}{\rho^3}
 +O_k(\rho^{-4}).
\label{eq:1}
\end{equation}

Since $h=np+e(K)$, substitute
\begin{equation}
\label{eq:2}
\rho=\sqrt h+c+d/\sqrt h+O_k(h^{-1})
\end{equation}
into
\eqref{eq:1}. By means of the Taylor expansion for $\frac{1}{\rho}$ (resp. $\frac{1}{\rho^2}$ and $\frac{1}{\rho^3}$).
The explicit expansions obtained from the Taylor series are as follows:
\[
\rho = \sqrt{h} \left( 1 + \frac{c}{\sqrt{h}} + \frac{d}{h} + O(h^{-3/2}) \right).
\]

Let
\(
t = \frac{c}{\sqrt{h}} + \frac{d}{h} + O(h^{-3}/2).
\)
Then
\[
\frac{1}{\rho} = \frac{1}{\sqrt{h}} \cdot \frac{1}{1+t} = \frac{1}{\sqrt{h}} \left( 1 - t + t^2 - t^3 + \cdots \right),
\ \
\frac{1}{\rho^2} = \frac{1}{h} (1+t)^{-2} = \frac{1}{h} \left( 1 - 2t + 3t^2 - 4t^3 + \cdots \right),
\]
and
\[
\frac{1}{\rho^3} = \frac{1}{h^{3/2}} (1+t)^{-3} = \frac{1}{h^{3/2}} \left( 1 - 3t + 6t^2 - 10t^3 + \cdots \right).
\]
Finally, by comparing the coefficients of the \(\sqrt{h}\)-term and the constant term on both sides of (\ref{eq:1}) and (\ref{eq:2}), gives
\[
 c=\frac{e(K[T])}{p},
\qquad
 d=\frac12\left(
 \frac{\sum_{z\in K}d_T(z)^2}{p}-e(K)-3c^2
 \right),
\]
which is~\eqref{eq:resolvent-expansion}. 
\end{proof}
Recall that $T$ is the common neighborhood of vertices in $K'$ and $T\subseteq K$ with $|T|=p$.
\begin{lemma}
\label{lem:T-clique}
It holds that \(p=k\) and \( K[T]\cong K_k.\)
\end{lemma}
\begin{proof}
By Lemma~\ref{lem:one-type},
\[
 \frac{e(K[T])}{p}\le\frac{k-1}{2}.
\]
If the inequality is strict, the constant term in
\eqref{eq:resolvent-expansion} is strictly less than the constant term
in~\eqref{eq:candidate-expansion}.  Hence equality is necessary, so
$K[T]$ is $(k-1)$-regular and $p\ge k$.

Put $J=K\setminus T$ and $d_z=d_T(z)$ for $z\in J$.  The coefficient
of $h^{-1/2}$ in~\eqref{eq:resolvent-expansion} is
\begin{equation}
 \beta(K,T)=
 \frac{(k-1)(k-1-2p)}8
 -\frac1{2p}\sum_{z\in J}d_z(p-d_z)
 -\frac12e(K[J]).
\label{eq:beta-general}
\end{equation}
Suppose $p\ge k+1$, we have
\[
\beta(K,T)\leq\frac{(k-1)(-k-3)}8=-(\frac{k^2-1}{8}+\frac{k-1}{4})<-(\frac{k^2-1}{8}+\frac{r(k-r)}{2k})
\]
for $k\ge3$ and \(0\leq r\leq k\).  Thus $\beta(K,T)$ in~\eqref{eq:beta-general} is strictly smaller than
the candidate coefficient in~\eqref{eq:candidate-expansion}. When \(h\) is sufficiently large, we have \(\rho<s_k(h)\), a
contradiction. Therefore $p=k$, and $(k-1)$-regularity gives
$K[T]=K_k$.
\end{proof}

The final optimization is purely modular, we called it the modular defect inequality.
\begin{lemma}
\label{lem:modular}
Let $0\le d_1,\ldots,d_s\le k$ and $e\ge0$ be integers satisfying
\(
 d_1+\cdots+d_s+e\equiv r\pmod k\) with
\( 1\le r\le k-1.
\)
Then
\begin{equation}
 \sum_{i=1}^s d_i(k-d_i)+ke\ge r(k-r).
\label{eq:modular}
\end{equation}
Equality holds if and only if $e=0$ and, after deleting terms equal to
$0$ or $k$, exactly one term remains and it equals $r$.
\end{lemma}

\begin{proof}
Let $f(t)=t(k-t)$.  For $0\le a,b\le k$, one has
\(
 f(a)+f(b)-f(a+b)=2ab\ge0\)
when \(a+b\le k,
\)
and
\(
 f(a)+f(b)-f(a+b-k)=2(k-a)(k-b)\ge0\) when
\(a+b > k.\)

Note \(f(0)=f(k)=0\), we have \(f(a)+f(b)\geq f(r')\), where $r'$ is the remainder upon dividing $a+b$ by $k$. Equality holds if and only if \(a=0\) or \(b=0\) (when \(a+b\leq k\)), and \(a=k\) or \(b=k\) (when \(a+b> k\)). That is to say, merging two residues modulo $k$ never increases total cost.
Treat each of the $e$ internal edges as a residue-one item of cost $k$;
this cost is strictly greater than $f(1)=k-1$.  Repeated merging gives
\begin{align}
\sum_{i=1}^s d_i(k-d_i)+ke&\geq f(d_1)+\cdots f(d_s)+f(1)+\cdots+f(1)\notag\\
                          &=f(d_1)+\cdots f(d_s)+e(k-1)\notag\\
                          &\geq f(r)\tag{$d_1+\cdots +d_s+e\equiv r \pmod k$}\\
                          &=r(k-r),\notag
\end{align}
so we obtain
\eqref{eq:modular}. Equality holds if and only if equality is attained in every step of the summation, i.e., \(d_i = k\) or \(d_i = 0\), after deleting terms equal to $0$ or $k$, exactly one term remains and it equals $r$ and \(e=0\).
\end{proof}
Recall that $H$ is $W_{2k+1}$-free and it is an $(\varepsilon,k)$-core satisfying $e(H)=h\to\infty$, $h-c_k\equiv r\pmod k$ with $1\le r\le k-1$.
\begin{proposition}
\label{prop:finite-extremum}
It holds
\(
 H\cong S_{k,h}
\)
up to isolated vertices, and $\rho(H)=s_k(h)$.
\end{proposition}
\begin{proof}
By Lemmas~\ref{lem:bounded-core}--\ref{lem:T-clique}, $H$ consists of
a bounded graph $K$, a large independent twin class $K'$ with common
neighborhood $T$, and
\(
 |T|=k,\ K[T]=K_k.
\)
Let $J=K\setminus T$ and put $d_z=d_T(z)$.  Since
\(
 h=nk+c_k+\sum_{z\in J}d_z+e(K[J]),
\)
we have
\[
 \sum_{z\in J}d_z+e(K[J])\equiv r\pmod k.
\]
Note that now with $p=k$. Hence Formula~\eqref{eq:beta-general} shows that the
second-order penalty is
\[
 \frac1{2k}\left(
 \sum_{z\in J}d_z(k-d_z)+k e(K[J])
 \right).
\]
By Lemma~\ref{lem:modular}, this is at least
$r(k-r)/(2k)$, exactly the penalty of $S_{k,h}$ in
\eqref{eq:candidate-expansion}.  All possible bounded pairs $(K,T)$
form a finite family.  Hence every strict inequality produces a fixed
positive gap in the coefficient of $h^{-1/2}$, which dominates the
uniform $O_k(h^{-1})$ remainder in
Lemma~\ref{lem:resolvent}.

Equality in Lemma~\ref{lem:modular} forces $e(K[J])=0$ and exactly one
vertex of partial $T$-degree $r$; moreover, every remaining vertex \( z \in J \) satisfies \( d_T(z) = k \) or \( d_T(z) = 0 \). Note that each vertex in $\{x|x\in J, d_T(x)=0\}$ is isolated.  Thus the nontrivial component is exactly
$S_{k,h}$.
\end{proof}

Now we are ready to prove our main result.
\begin{proof}[\bf Proof of Theorem~\ref{thm:main}]
The case $r=0$ is Theorem~\ref{thm:FZZ}, so assume $1\le r\le k-1$.
Suppose that an $m$-edge $W_{2k+1}$-free graph $G$
satisfies
\(
 \rho(G)\ge s_k(m).
\)
Apply Lemma~\ref{lem:core-extraction} to obtain an
$(\varepsilon,k)$-core $H$ with
\[
 h=(1-o(1))m,\qquad h\equiv m\pmod k,\qquad
 \rho(H)\ge s_k(h).
\]
Proposition~\ref{prop:finite-extremum} gives
$H\cong S_{k,h}$ and $\rho(H)=s_k(h)$.  The strict part of
Lemma~\ref{lem:core-extraction} now implies $H=G$.  Consequently
$h=m$ and $G\cong S_{k,m}$.

Conversely, Proposition~\ref{prop:candidate} shows that $S_{k,m}$ is
$W_{2k+1}$-free and has $m$ edges.  This proves the upper bound and the
equality characterization.
\end{proof}

\section{Concluding remarks}\label{s4}

The congruence restriction in the dense-core definition is essential.
For a fixed residue $r$,
\[
 \frac{s_{k,r}(m)}{\sqrt m}
 =1+\frac{k-1}{2\sqrt m}
 -\left(\frac{k^2-1}{8}+\frac{r(k-r)}{2k}\right)m^{-1}
 +O_k(m^{-3/2}),
\]
whose derivative is $O_k(m^{-3/2})$.  Across consecutive integers,
however, the residue-dependent coefficient jumps by order $m^{-1}$,
which is too large for a one-edge dense-core comparison.  Deleting a
multiple of $k$ edges is precisely what keeps the target on one smooth
branch.

The finite-core portion of the argument is related in spirit to the
bounded-outer-layer and twin-class method of Das and Yamini
\cite{DasYamini}.  Here the local $C_{2k}$ obstruction supplies the
degree bound on the dominant neighborhood type, while the residue class
is resolved by the convex modular cost $d(k-d)$.

We also notice that Yu, Zhang and Zhang~\cite{YZZ} proposed another conjecture, which is weaker than Conjecture~\ref{conj:1.4}. We list it as follows:
\begin{conjecture}[Yu, Zhang and Zhang~\cite{YZZ}]\label{c5}
Let $k\ge 2$ be fixed and $m$ be sufficiently large, and $G$ is a $\{W_t: t\ge 2k+1\}$-free graph with $m$ edges.
\begin{enumerate}[label=\rm(\roman*)]
\item If $m=\binom{k}{2}+kq,
\) then
\(
\rho(G)\le \frac{k-1+\sqrt{4m-k^2+1}}{2}.
\)
Equality holds if and only if
\(
 G\cong K_k\vee qK_1,$
up to isolated vertices.

\item If $m=\binom{k}{2}+kq+r$ with $1\le r\le k-1,$ then
\(
\rho(G)\le \rho(S_{k,m}).
\)
Equality holds if and only if
$G\cong S_{k,m}$, where $S_{k,m}$ is obtained from $K_k\vee qK_1$ by adding a vertex $z$ and joining it to exactly $r$ vertices of the
$K_k$.
\end{enumerate}
\end{conjecture}
By Theorems~\ref{thm:FZZ} and \ref{thm:main}, Conjecture~\ref{c5} holds for $k\ge 3.$ For the case $k=2,$ one may see Conjecture~\ref{c5} does not hold based on the results in \cite{LGY}.

\section*{Statements and Declarations}

\textbf{Acknowledgments}
We would like to express our sincere gratitude to Dr. Yongtao Li for his careful reading of this manuscript and for offering insightful comments and constructive suggestions, which substantially improved the presentation and rigor of our work.

\textbf{Competing interests}
The authors declare that they have no competing interests.

\textbf{Data availability}
No data were generated or analysed during the current study.

\textbf{Funding}
Shuchao Li was supported by the National Natural Science Foundation of China (Grant Nos. 12571365, 12171190).


\textbf{Use of AI tools}
During the preparation of this manuscript, the authors used Doubao for language polishing and structural organization. The authors reviewed and edited all content and take full responsibility for the final manuscript.


\begin{thebibliography}{99}
\small \setlength{\itemsep}{-.8mm}
\bibitem{BM2008}
J.A. Bondy and U.S.R. Murty, \emph{Graph Theory},
Graduate Texts in Mathematics, Vol.~244, Springer, New York, 2008.

\bibitem{BrualdiHoffman}
R.A. Brualdi and A.J. Hoffman,
On the spectral radius of $(0,1)$-matrices,
\emph{Linear Algebra Appl.} \textbf{65} (1985), 133--146.

\bibitem{Liu-Li-Li-Yu}
L. Chang, J.P. Li, S.C. Liu, Y.T. Yu,
A Brualdi-Hoffman-Turan problem on theta graph,
\emph{Adv. in Appl. Math.} \textbf{173} (2026) 103000.

\bibitem{CioabaDesaiTait}
S.~M. Cioab\u{a}, D.~N. Desai and M. Tait,
The spectral radius of graphs with no odd wheels,
\emph{European J. Combin.} \textbf{99} (2022), Paper No.~103420.

\bibitem{DasYamini}
J. Das and V. Yamini,
A sharp fixed-size spectral bound for \(kK_3\)-free graphs,
\url{arXiv:2608.05869, 2026.}

\bibitem{FangLinZhai}
L.~Fang, H.~Lin and M.~Zhai,
Edge-spectra supersaturation for tripartite color-critical graphs,
\url{arXiv:2608.04485, 2026.}


\bibitem{FangZhaiZhang}
L.~Fang, M.~Zhai and Y.~Zhang,
Dense-core approach to the Brualdi--Hoffman--Tur\'an problem on odd wheels,
\url{arXiv:2608.16127, 2026.}

\bibitem{LGY}
J. Gao, Xianya Geng, S.C. Li, Spectral extremal graphs for $W_5$-free graphs with odd size, in preparation.


\bibitem{GaoLi}
J. Gao, X. Li, Spectral radius of graphs of given size with a forbidden fan graph $F_6$,
\emph{Discrete Math.} \textbf{349} (2026) 114695.

\bibitem{Godsil1}
C. Godsil, G. Royle,
\emph{Algebraic Graph Theory}, vol. 207 of Graduate Texts in Mathematics,
Springer-Verlag, New York, 2001.

\bibitem{H-J-2012}
R.A. Horn and C.R. Johnson, \emph{Matrix Analysis}, 2nd ed.,
Cambridge University Press, Cambridge, 2012.

\bibitem{JoyentanujYamini}
D.~Joyentanuj and V.~Yamini,
A sharp fixed-size spectral bound for $kK_3$-free graphs,
\url{arXiv:2608.05869, 2026.}


\bibitem{LiZhaoZou}
S.C.~Li, S.S.~Zhao and L.T.~Zou,
Spectral extrema of graphs with fixed size: forbidden a fan graph, a friendship graph or a theta graph,
\emph{J. Graph Theory} \textbf{110} (2025), no.~4, 483--495.

\bibitem{LiLiuZhangStability}
Y.~Li, H.~Liu and S.~Zhang,
An edge-spectral Erd\H{o}s--Stone--Simonovits theorem and its stability,
\url{arXiv:2508.15271, 2025.}


\bibitem{LiZhaiShu}
X.~Li, M.~Zhai and J.~Shu,
A Brualdi--Hoffman--Tur\'an problem on cycles,
\emph{European J. Combin.} \textbf{120} (2024), Paper No.~103966,
13 pp.

\bibitem{N1}V.~Nikiforov, Some inequalities for the largest eigenvalue of a graph, \emph{Combin. Probab. Comput.} \textbf{11} (2002) 179--189.

\bibitem{E1}E.~Nosal, \emph{Eigenvalues of Graphs}, Master's thesis, University of Calgary, 1970.

\bibitem{YLP}
L.J. Yu, Y.T. Li, Y.J. Peng,
Spectral extremal graphs for fan graphs, \emph{Discrete Math.} \textbf{348} (5) (2025) 114391.

\bibitem{YZZ}Y.T. Yu, H.H. Zhang, M.J. Zhang, A survey of edge-spectral-Tur\'an type problems in spectral graph theory: Results, conjectures and open problems, \emph{Discrete Appl. Math.} \textbf{395} (2026) 210-221.

\bibitem{ZhaiLiLou}
M. Zhai, R. Li and Z. Lou,
Advances on two spectral conjectures regarding booksize of graphs,
\emph{European J. Combin.} \textbf{138} (2026), Paper No.~104431.

\bibitem{ZhaiLinShu}
M.~Zhai, H.~Lin and J.~Shu,
Spectral extrema of graphs with fixed size: cycles and complete bipartite graphs,
\emph{European J. Combin.} \textbf{95} (2021), Paper No.~103322,
18 pp.

\bibitem{ZhangOddWheels}
W. Zhang,
More results on the spectral radius of graphs with no odd wheels,
arXiv:2408.03595, 2024.

\bibitem{ZhangWang}
Y.~Zhang and L.~Wang,
Spectral extrema of graphs with fixed size: forbidden star forests,
\emph{Discrete Math.} \textbf{349} (2026), no.~5,
Paper No.~114976, 10 pp.


\end{thebibliography}
\end{document}